\documentclass{cimart}

\usepackage{amsmath}
\usepackage{amssymb}
\usepackage{accents}

\usepackage{mathtools}
\usepackage{amsthm}

\usepackage{graphicx}

\newcommand{\1}[1]{{\mathbf 1}{\{#1\}}}

\newcommand{\eps}{\varepsilon} 
\newcommand{\Z}{{\mathbb Z}}

\newcommand{\C}{{\mathbb C}}
\newcommand{\B}{{\mathsf B}}

\newcommand{\rad}{{\mathop{\mathrm{rad}}}}

\newcommand{\am}{{\mathfrak A}}
\newcommand{\cl}{{\mathfrak C}}

\newcommand{\Sph}{{\mathbb S}}
\newcommand{\V}{{\mathcal V}}

\newcommand{\LL}{{\mathcal L}}

\newcommand{\R}{{\mathbb R}}
\newcommand{\RR}{{\mathcal R}}

\newcommand{\tell}{{\tilde\ell}}

\newcommand{\BRI}{\mathop{\mathrm{BRI}}}

\let\phi=\varphi

\newcommand{\tW}{{\widetilde W}}
\newcommand{\tY}{{\widetilde Y}}

\newcommand{\eqlaw}{\stackrel{\text{\rm\tiny law}}{=}}

\newcommand{\IP}{{\mathbb P}}
\newcommand{\IE}{{\mathbb E}}

\newcommand{\distTV}{\mathop{dist}\nolimits_{\mathit{TV}}}

\DeclareMathSymbol{\widehatsym}{\mathord}{largesymbols}{"62}

\newcommand{\hW}{\widehat{W}}

\newcommand{\capa}{\mathop{\mathrm{cap}}}

\newcommand{\hcapa}{\mathop{\widehat{\mathrm{cap}}}}
\newcommand{\hm}{\mathop{\mathrm{hm}}\nolimits}

\newcommand{\htau}{\widehat{\tau}}

\allowdisplaybreaks

\newtheorem{theo}{Theorem}[section]
\newtheorem{lem}[theo]{Lemma}
\newtheorem{df}[theo]{Definition}
\newtheorem{prop}[theo]{Proposition}
\newtheorem{cor}[theo]{Corollary}

\title{
   On the shape of the connected components
of the complement 
of two-dimensional Brownian random interlacements
    }

\author{
    Orph\'ee Collin and Serguei Popov
    }

\authorinfo[
    O.\ Collin]{
  CNRS, LPSM,  Paris Cit\'e University  
and Sorbonne University, France}{%
    collin@lpsm.paris
    }

\authorinfo[
    S.\ Popov]{
   CMUP,  University of Porto, Portugal}{%
    serguei.popov@fc.up.pt
    }

\abstract{%
    We study 
the limiting shape of the connected components
 of the vacant set of
two-dimensional Brownian random interlacements: 
we prove that the connected component around~$x$
is close in distribution to
a rescaled \emph{Brownian amoeba} in the regime
when the distance from~$x\in\C$ to the closest trajectory
is small (which, in particular, includes the cases
 $x\to\infty$ with fixed intensity parameter~$\alpha$, 
 and $\alpha\to\infty$ with fixed~$x$).
We also obtain a new family of martingales
built on the conditioned Brownian motion,
which may be of independent interest.
    }



\VOLUME{33}
\YEAR{2025}
\ISSUE{1}
\NUMBER{7}
\DOI{https://doi.org/10.46298/cm.14455}

\begin{document}


\section{Introduction}
\label{s_intro}
In this paper, we study the model of Brownian random interlacements
in two dimensions. The discrete random interlacements
in higher dimensions are ``canonical
Poissonian soups'' of doubly infinite
simple random walk's trajectories
(or ``loops passing through infinity'');
they were introduced
by Sznitman in~\cite{S10} and then studied quite extensively (see, e.g.~\cite{DRS14}). That original construction
cannot be used in low dimensions where a simple random walk's trajectory
is recurrent so that even a single trajectory
will a.s. fill the whole space;
one can, however, construct a similar object in a natural
way using the \emph{conditioned} (on avoiding the origin) trajectories.
This was done in~\cite{CPV16} in two dimensions and in~\cite{CDarcyP18}
in one dimension. In all dimensions, random interlacements are related to (simple) random walks
on tori: the traces left by the random walk up to a
certain time on a box (with a size much smaller than that of the torus) can be well approximated by the random
interlacements restricted to that box (in dimensions~$1$ and~$2$ one also needs to condition the center of the
box to be not visited by the walk).
 
It is also natural to consider a continuous space/time counterpart
of the above-mentioned models:
namely, the ``interlacement soup''
made of Brownian trajectories
(which, in the case of lower dimensions,
also needs some conditioning). This was done in~\cite{Szn13}
for the original random interlacement model and in~\cite{CP20}
for the two-dimensional case. 

One of the main objects of interest in the context of random interlacements
is \emph{the vacant set}, which is the complement
of the union of the trajectories that make up
the interlacements. 
This article is part of a program focused on studying the geometric properties of the vacant set
of the two-dimensional random interlacements
(discrete and continuous), as well as the vacant
set of a single-conditioned trajectory.
Note that in the two-dimensional case,
the vacant set of Brownian interlacements
only has bounded connected
components; the paper~\cite{CP20} contains some
results mainly related to the (linear) size
of these components. In particular, it was shown there
(see also~\cite{CP24}) that the vacant set
contains an infinite number of nonoverlapping disks
of constant radius whenever $\alpha\leq 1$,
where~$\alpha$ is the intensity parameter of the model.
In addition, the radius of the largest 
disk around~$x$, which is fully contained in the vacant
set is $\exp(-Z\alpha\ln^2|x|)$, where~$Z$ is 
a random variable with some (explicit)
distribution.
In Theorem~\ref{t_conv_Poisson} below, we extend
that last result by studying not only the distance
to the closest trajectory but also
to the second closest one, and so on,
but the main focus of this paper is to
study the geometric shapes of the connected
components of the vacant set. We mention 
that the connected components of the complement
of \emph{one} The Brownian trajectory was studied
in a number of papers, notably (in chronological order)
\cite{M89,LG91,Wer94,HNPS19}; in particular,
the limiting shape was the main subject of~\cite{Wer94}.
It is interesting to note that, despite the fact that we are dealing
with many trajectories here, in some regimes
it is still typically only one
trajectory that defines the shape
of a given connected component, 
cf.\ Theorem~\ref{t_amoeba} below.
Then, in Theorem~\ref{t_central_cell} we also consider
a situation when the shape of the connected component
is typically determined by many trajectories:
in the limit~$\alpha\to\infty$, 
we prove some geometric properties of the ``central
cell'' (that is, the connected component which contains
the unit disk centered at the origin, which the trajectories are conditioned not to touch).

Another contribution of this paper we have to mention
is the following: in Section~\ref{s_further_cond_BM}, 
we present a family of functions
which, when applied to the conditioned Brownian
motion, result in (local) martingales
(see Proposition~\ref{p_L_properties} below).
That can be seen as a continuous counterpart 
of Proposition~2.4 of~\cite{P21}
or Proposition~4.10 of~\cite{P21_book},
where the corresponding families for the simple conditioned two-dimensional random walk
were discussed. It is generally very useful
to be able to construct such martingales since it allows, e.g., estimating hitting probabilities
of various sets via the optional stopping theorem;
therefore, we hope that Proposition~\ref{p_L_properties}
will find its further applications.


\section{Formal definitions and results}
\label{s_def_results}
In the following, we will identify~$\R^2$ and~$\C$ 
via $x=(x_1,x_2)=x_1+i x_2$,
$|\cdot|$ will denote the Euclidean norm
in~$\R^2$ 
as well as the modulus in~$\C$.
Also, let $\B(x,r)=\{y: |x-y|\leq r\}$ be the
closed disk of radius~$r$ centered in~$x$,
and abbreviate $\B(r):=\B(0,r)$. 

Let~$W$ be the standard two-dimensional Brownian motion. 
The main ingredient of Brownian random interlacement is
the Brownian motion is conditioned on not hitting
the unit disk~$\B(1)$,
which will be denoted by~$\hW$; 
formally, it is the Doob's $h$-transform 
of~$W$ with respect to $h(x)=\ln|x|$.
The process~$\hW$ can be formally defined 
via its transition kernel~$\hat p$: 
for  $|x|>1, |y|\geq 1$, 
\begin{equation}
\label{df_hat_p}
 {\hat p}(t,x,y) = p_0(t,x,y)\frac{\ln|y|}{\ln|x|} .
\end{equation}
where~$p_0$ denotes
the transition subprobability density of~$W$ killed on hitting
 the unit disk~$\B(1)$. 
It is possible to show (see~\cite{CP20}) that the diffusion~$\hW$
obeys the stochastic differential equation
\begin{equation}
\label{differential_W}
 d\hW_t = \frac{\hW_t}{|\hW_t|^2\ln |\hW_t|} dt
    +dW_t.
\end{equation}
Sometimes, it can be useful to work with an alternative definition of 
the diffusion~$\hW$ using polar coordinates
$\hW_t = (\RR_t\cos \Theta_t, \RR_t\sin \Theta_t)$. With~$W^{(1,2)}$
two independent standard one-dimensional 
Brownian motions,
let us consider the stochastic differential equations
\begin{align}
 d \RR_t &= \Big(\frac{1}{\RR_t\ln \RR_t} + \frac{1}{2\RR_t}\Big)dt 
 + dW_t^{(1)},
   \label{df_Rt} \\ 
\label{df_Theta_t}
d \Theta_t &=  \frac{1}{\RR_t} dW_t^{(2)}
\end{align}
(note that the diffusion~$\Theta$ takes values in the whole~$\R$,
so we are considering
a Brownian motion on the Riemann surface);
it is an elementary exercise in stochastic calculus to show
that~\eqref{differential_W} is equivalent
to~\eqref{df_Rt}--\eqref{df_Theta_t}. 
We also note that, as shown in~\cite{CP20}, even though the radial
drift component in the above stochastic differential equations is not
well defined $\partial\B(1)$, it is still possible to define the 
process~$\hW$ starting at~$\partial\B(1)$ and staying outside
the unit disk for all $t>0$.

Let us define
\begin{align*}
 \tau(x,r) &= \inf\big\{t>0: W_t\in\partial\B(x,r)\big\},\\
 \htau(x,r) &= \inf\big\{t>0: \hW_t\in\partial\B(x,r)\big\}
\end{align*}
to be the hitting times of the boundary of the disk~$\B(x,r)$ 
with respect to the (two-dimensional)
Brownian motion and its conditioned version;
we abbreviate $\tau(r):=\tau(0,r)$ and 
$\htau(r):=\htau(0,r)$.

Recall the definition of 
\emph{Wiener moustache} (Definition~2.4
of~\cite{CP20}):
\begin{df}
\label{df_moustache}
 Let~$U$ be a random variable with a uniform distribution in $[0,2\pi]$, and let
$(\RR^{(1,2)},\Theta^{(1,2)})$
be two independent copies of the processes
defined by \eqref{df_Rt}--\eqref{df_Theta_t}, 
with a common initial point~$(1,U)$.
Then, the \emph{Wiener moustache}~$\eta$ 
is defined as the union of ranges of the two trajectories, that is,
\[
\eta = \big\{re^{i\theta}: \text{ there exist }k\in\{1,2\}, t\geq 0
 \text{ such that } \RR^{(k)}_t=r, \Theta^{(k)}_t=\theta\big\}. 
\]
\end{df}

We also need to recall
the definition of capacity for subsets of~$\R^2$. 
Let~$A$ be a compact subset of~$\R^2$
such that $\B(1)\subset A$. 
Denote by~$\hm_A$ the \emph{harmonic measure} 
(from infinity) on~$A$, that is, the entrance law in~$A$ for the Brownian motion starting from infinity (cf.\ e.g.\ Theorem~3.46 of~\cite{MP10}). 
We define the (Brownian) 
capacity of~$A$ as
\begin{equation}
\label{df_Br_cap}
 \capa(A) = \frac{2}{\pi}\int_A \ln|y|\, d\hm_A(y) .
\end{equation}
Also, for \emph{any} compact subset~$A$ of~$\R^2$, 
we define $\hcapa(A):=\capa\big(\B(1)\cup A\big)$.

Now, we recall the definition of two-dimensional
Brownian random interlacements
\cite{CP20}:
\begin{df}
\label{df_BRI}
 Let~$\alpha>0$ and consider a Poisson point process
$(\rho_k^\alpha, k\in\Z)$ on~$\R_+$ 
with intensity~$r(\rho)=\frac{2\alpha}{\rho}, 
\rho\in\R_+$.
Let $(\eta_k, k\in\Z)$ be a sequence of i.i.d.\ Wiener moustaches.
Fix $b\geq 0$.
Then, the model of Brownian Random Interlacements (BRI) 
on level~$\alpha$
 truncated at~$b$ 
is defined in the following way:
\begin{equation}
\label{eq_BRI}
\BRI(\alpha;b) = \bigcup_{k:\rho_k^\alpha \geq b}\rho_k^\alpha \eta_k\,.
\end{equation}
We also define the \emph{vacant set}~$\V^{\alpha;b}$: it is
the set of points of the plane that do not
belong to trajectories of $\BRI(\alpha;b)$
\[
\V^{\alpha;b} = \R^2\setminus \BRI(\alpha;b).
\]
\end{df}
Let us abbreviate $\BRI(\alpha):=\BRI(\alpha;1)$
and $\V^{\alpha}:=\V^{\alpha;1}$.
It is a characteristic property of Brownian random interlacements
that (recall Proposition~2.11 of~\cite{CP20})
\begin{equation}
\label{eq_vacant_Bro}
 \IP\big[A\cap\BRI(\alpha)=\emptyset\big]
    =  \exp\big(-\pi\alpha \hcapa(A)\big).
\end{equation}

An important observation is that the above Poisson process
is the image of a homogeneous 
 Poisson process of rate~$1$ in~$\R$ under the map 
$x\mapsto e^{x/2\alpha}$;
this follows from the mapping theorem
for Poisson processes (see e.g.\ Section~2.3 of~\cite{K93}).
Because of that, we may write
\begin{equation}
\label{rho_exponential}
 \rho_k^\alpha = \exp\Big(\frac{Y_1+\cdots+Y_k}{2\alpha}\Big),
\end{equation}
where $Y_1,\ldots,Y_k$ are i.i.d.\ Exponential(1) 
random variables. Also, as mentioned in Remark~2.7
of~\cite{CP20}, one can actually construct~$\BRI(\alpha)$
for all $\alpha>0$ simultaneously,
in such a way that $\BRI(\alpha_1)$
\emph{dominates} $\BRI(\alpha_2)$ for $\alpha_1>\alpha_2$:
for this, one can consider a Poisson process of rate~$1$
in~$\R^2_+$ (with coordinates $(\rho,u)$)
and then take those points that lie
below the curve $u=\frac{2\alpha}{\rho}$
when constructing~$\BRI(\alpha)$.

It is clear that the above construction is
not invariant with respect to translations
of~$\R^2$.
Let us also mention an equivalent construction 
which, in some sense, recovers the translation 
invariance property (the random interlacement
is obtained as an image of an object which is ``more
translationally invariant'').
For that, let us first observe that the following
fact holds (recall that the Bessel process of dimension~3, 
also denoted here as Bes(3), is the norm of the 
three-dimensional standard Brownian motion):
\begin{prop}
\label{p_represent_hW}
 Let $\hW$ be the conditioned Brownian motion
started somewhere outside $\B(1)$ (or on its boundary).
Then, there exists a pair of independent processes
$(Z,B)$, where $Z$ is Bes(3) and $B$ is a Brownian
motion such that
\begin{equation}
\label{eq_represent_hW} 
\hW_t = \exp(Z_{G_t}+iB_{G_t}), \text{ where }
 G_t = \int_0^t\frac{ds}{|\hW_s|^2}.
\end{equation}

\end{prop}
\begin{proof}
 This follows from the skew-product 
representation of the Brownian motion, cf.\
e.g.\ Theorem~7.19 of~\cite{LeGall16},
together with the well-known fact that Bes(3)
is the (one-dimensional) Brownian motion conditioned
on never hitting the origin (cf.~\cite{McK63,Will74}).
\end{proof}

Now, let $\Sph^1_{2\pi}=\R/2\pi\Z$ be 
the circle of radius~$1$;
then, one can naturally
define the exponential map from~$\R\times \Sph^1_{2\pi}$
to~$\C$ by $\exp(r,\theta)=\exp(r+i\theta)$.
For $(r,\theta)\in\R\times \Sph^1_{2\pi}$, we
call a Bessel moustache (attached to that point)
a pair of independent trajectories
$(r+Z^{(1,2)}_t, \theta+B^{(1,2)}_t, t\geq 0)$,
where $Z^{(k)}$ is a Bes(3) process (taking values
in~$\R_+$ and starting at~$0$)
and $B^{(k)}$ is an independent Brownian motion
on~$\Sph^1_{2\pi}$ starting at~$0$. 

Then, due to~\eqref{rho_exponential} 
and Proposition~\ref{p_represent_hW},
we can define $\BRI(\alpha)$ in the following
way\footnote{In fact, we will obtain the same model if 
the interlacement trajectories are viewed 
as subsets of~$\C$; note the time change 
in~\eqref{eq_represent_hW}.}
(see Figure~\ref{f_Bessel_x_BM}):
\begin{itemize}
 \item take a Poisson point process 
of rate $2\alpha\times \frac{1}{2\pi}$ 
on $\R_+\times \Sph^1_{2\pi}$ (i.e.,
take a one-dimensional Poisson process 
with rate~$2\alpha$ and rotate the points 
randomly on $\Sph^1_{2\pi}$);
 \item attach Bessel moustaches to these points
 independently;
 \item transfer that picture to~$\C$ using the 
 exponential map.
\end{itemize}
\begin{figure}
\begin{center}
\includegraphics{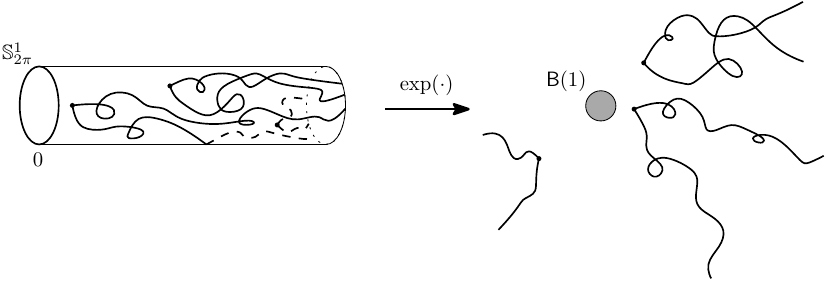}
\caption{On the equivalent definition of $\BRI(\alpha)$.}
\label{f_Bessel_x_BM}
\end{center}
\end{figure}
We also remark that in a similar way, one can obtain
the process $\BRI(\alpha;b)$ 
by taking a Poisson process on 
$[\ln b,+\infty)\times \Sph^1_{2\pi}$
in the above construction.

Now, we are almost ready to state our results,
but we still need to recall 
a couple of notations from~\cite{CP20}:
for $x\notin\B(1)$,
\[
\ell_x = \int_{\partial \B(1)}
\ln|x-z| \, H(x,dz)
= \frac{|x|^2-1}{2\pi}\int_{\partial \B(1)}
   \frac{\ln|x-z|}{|x-z|^2}\, dz,
\]
where $H(x,\cdot)$ is the entrance measure (with respect to
the Brownian motion) to~$\B(1)$ starting from~$x$
(see~\eqref{Poisson_kernel} below).
As argued in (3.9)--(3.10) of~\cite{CP20}, 
\begin{equation}
\label{behaviour_ell_x}
\ell_x = \big(1+O(|x|^{-1})\big) \ln |x| \text{ as }|x|\to\infty
\text{ and } \ell_x = \ln (|x|-1)  + O(1) \text{ as }|x|\downarrow 1.
\end{equation}
Let us also define (for the specific purpose of being inside $O$'s, given that~$\ell_\cdot$ changes sign and can be equal to~$0$)
$\tell_x:=|\ell_x|\vee 1$.

We recall another notation from~\cite{CP20}: for $x\in\R^2$,
$\Phi_x(\alpha)$ denotes the distance from~$x$ to the closest
trajectory of~$\BRI(\alpha)$. We now extend this by defining
$\Phi_x^{(0)}(\alpha):=0$,
$\Phi_x^{(1)}(\alpha):=\Phi_x(\alpha)$, $\Phi_x^{(2)}(\alpha)$ 
to be the distance to the second closest trajectory, 
$\Phi_x^{(3)}(\alpha)$ 
to be the distance to the third closest trajectory, 
and so on; note that a.s.\ it holds that
$0<\Phi_x^{(1)}(\alpha)<\Phi_x^{(2)}(\alpha)<\Phi_x^{(3)}(\alpha)<\ldots$. 
Let us denote also 
\[
\tY_x^{(1)}(\alpha) = \frac{2\alpha\ln^2|x|}{\ln(\Phi_x^{(1)}(\alpha)^{-1})}, \quad
\tY_x^{(j)}(\alpha) = \frac{2\alpha\ln^2|x|}{\ln(\Phi_x^{(j)}(\alpha)^{-1})}
    - \frac{2\alpha\ln^2|x|}{\ln(\Phi_x^{(j-1)}(\alpha)^{-1})},\quad
    j\geq 2.
\]
The following result states 
that~$\tY_x^{(m)}(\alpha)$ are approximately Exponential($1$)
for $m\geq 1$, and also that $\tY_x^{(j+1)}(\alpha)$
is approximately independent of
$\tY_x^{(1)}(\alpha),\ldots,\tY_x^{(j)}(\alpha)$:
\begin{theo}
\label{t_conv_Poisson}
Assume that $x\notin\B(1)$, and let $b_1,\ldots,b_j>0$.
Then
\begin{equation}
\label{eq_conv_Poisson_Y1}
 \IP\big[\tY_x^{(1)}(\alpha)>s\big]
  = e^{-s}\big(1+s\times O\big(
\Psi_1^{(s)} + \Psi_2^{(s)} + \Psi_3^{(s)} \big)\big),
\end{equation}
and, for $j\geq 1$,
\begin{align}
\lefteqn{
 \IP\big[\tY_x^{(j+1)}(\alpha)>s
  \mid \tY_x^{(1)}(\alpha)=b_1,\ldots,\tY_x^{(j)}(\alpha)=b_j\big]
}
\nonumber\\
& = e^{-s}\big(1+s(b+s)\times O\big(
\Psi_1^{(b+s)} + \Psi_2^{(b)} + \Psi_2^{(b+s)} 
 +\Psi_3^{(b)} + \Psi_3^{(b+s)}
    \big)\big),
\label{eq_conv_Poisson} 
\end{align}
where $b:=b_1+\cdots+b_j$, and 
\begin{align*}
 \Psi_1^{(h)} &= 
 \frac{\tell_x h}
 {\alpha \ln^2|x|} , \\
 \Psi_2^{(h)} &= \frac{1+\alpha h^{-1}\ln^2|x|+\ln|x|}
  {\exp(2h^{-1}\alpha\ln^2|x|)|x|\ln|x|} ,\\
 \Psi_3^{(h)} &= \frac
       {1+\Psi_1^{(h)}}
    {\exp(2h^{-1}\alpha\ln^2|x|)(|x|-1)} .
\end{align*}
\end{theo}
The above result can be interpreted informally in the following way:
\[
 \Big(\frac{2\alpha\ln^2|x|}{\ln(\Phi_x^{(1)}(\alpha)^{-1})},
  \frac{2\alpha\ln^2|x|}{\ln(\Phi_x^{(2)}(\alpha)^{-1})},
   \frac{2\alpha\ln^2|x|}{\ln(\Phi_x^{(3)}(\alpha)^{-1})},
   \ldots \Big)
\]
is approximately a Poisson process of rate~$1$ in~$\R_+$, 
as long as the error term in~\eqref{eq_conv_Poisson}
is small.
(Note that the quantity inside $O(\dots)$
in~\eqref{eq_conv_Poisson} does not depend on~$j$.)
Let us also observe that the error term in~\eqref{eq_conv_Poisson}
is $O(\alpha^{-1})$ when~$x$ is fixed and $\alpha\to\infty$,
and is $O(\frac{1}{\ln|x|})$
when~$\alpha$ is fixed and $|x|\to\infty$.

We need to define another important object, which
is derived from the Wiener moustache. 
\begin{df}
\label{df_amoeba}
Let $\eta$ be a Wiener moustache. 
The \emph{Brownian amoeba}~$\am$
is the connected component of the origin in the complement of
$\eta \subset \R^2$,
see Figure~\ref{f_amoeba}. 
\end{df}
\begin{figure}
\begin{center}
\includegraphics{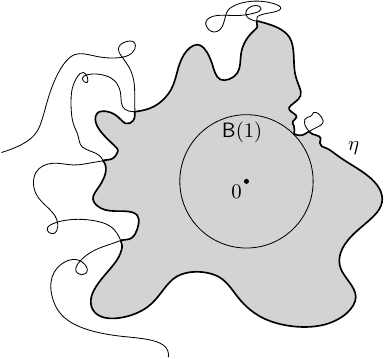}
\caption{The Brownian amoeba $\am$}
\label{f_amoeba}
\end{center}
\end{figure}
Observe that the Brownian amoeba a.s.\ contains $\B(1)$
(except for the point where the Wiener moustache touches the unit disk). 
To the best of our knowledge,
this object first appeared in~\cite{Wer94}
(it is the one with the distribution~$\LL_1$ there)
as the limiting shape of the connected components
of the complement of \emph{one} Brownian
trajectory.
A remarkable property of the Brownian amoeba
is that if we move the origin to a uniformly randomly
chosen (with respect to the area) point in~$\am$
and rescale it properly
(so that it touches the boundary
of the unit disk centered at the new origin), 
then the resulting object has the same law; 
this is Proposition~22 of~\cite{Wer94}.

Next, we formulate the main result of this
paper, which says that in certain regimes,
the connected components of the vacant set
converge in distribution to the Brownian amoeba.
Analogously to Definition~\ref{df_amoeba}, 
for $x\in\R^2$, define
\[
\cl_x(\alpha)=
\text{the connected component of } x \text{ in } \R^2 \setminus \BRI(\alpha) ,
\]
with the convention that $\cl_x(\alpha)=\emptyset$  if $x\in \BRI(\alpha)$.
In the next result, we show that 
 $\cl_0(\alpha)/\Phi_0(\alpha)$ converges to~$\am$
in total variation distance as $\alpha\to 0$, and also
 $\cl_x(\alpha)/\Phi_x(\alpha)$ converges to~$\am$
in total variation distance under
certain conditions (which we discuss in more detail below). In the following, $\distTV(X,Y)$
denotes the total variation distance between 
the laws of random objects~$X$ 
and~$Y$ (recall that this total variation distance 
is defined as $\inf\IP[X\neq Y]$, where the infimum
is taken over all couplings of $X$ and $Y$).
\begin{theo}
\label{t_amoeba}
\begin{itemize}
\item[(i)] We have for a positive constant~$c_1$
\begin{equation}
\label{eq_amoeba_0}
\distTV\Big(\frac{\cl_0(\alpha)}{\Phi_0(\alpha)}, \am\Big)
  \leq c_1\alpha.
\end{equation}
\item[(ii)] 
Assume that $|x| > 1$ and $\alpha \ln^2|x|\geq 2$.
 Then, for some, $c_2>0$ it holds that
\begin{equation}
\label{eq_amoeba_x}
\distTV\Big(\frac{\cl_x(\alpha) - x}{\Phi_x(\alpha)},\am \Big)
\leq c_2\frac{\tell_x \ln^3(\alpha \ln^2 |x|)}
   {\alpha \ln^2|x|}.
\end{equation}
\end{itemize}
\end{theo}
We are not sure how sharp
is the estimate~\eqref{eq_amoeba_x}.
In any case, note that for the term in
the right-hand side of~\eqref{eq_amoeba_x}
to be small,
we always need~$\alpha \ln^2|x|$ to be large, but that 
might not be enough for convergence to~$\am$ (since there 
is also the factor~$\tell_x$, which grows to infinity
as $|x|\to\infty$ or $|x|\downarrow 1$).
To give a few examples, we observe that
the error term in~\eqref{eq_amoeba_x} is
\begin{itemize}
 \item $O\big(\frac{\ln^3\alpha}{\alpha}\big)$ when~$x\notin\B(1)$ 
 is fixed and $\alpha\to\infty$;
 \item $O\big(\frac{(\ln\ln |x|)^3}{\ln |x|}\big)$ when~$\alpha>0$ 
 is fixed and $|x|\to\infty$;
 \item $O\big(\frac{|\ln v|\ln^3(\alpha v^2)}{\alpha v^2}\big)$ when
 $v:=|x|-1\downarrow 0$ (in this case, we need~$\alpha$ to be
 ``a bit larger'' than $(1/v)^2$, i.e., $\alpha=(1/v)^{2+\eps}$
 would work).
\end{itemize}
In any case, it is interesting to observe the similarity between
what is seen from the origin in the $\alpha\to 0$ regime
and what is seen from $x\notin \B(1)$ the ``high
intensity'' regime:
 only one trajectory (the one that forms
the amoeba) really ''matters''; others are much more distant.

\medskip

We also discuss the ``central cell'' $\cl_0(\alpha)$
in the regime when $\alpha\to \infty$.
In this situation, the cell is likely to be formed
by \emph{many} trajectories.
Let us denote by 
\begin{align}
 m_\alpha &= \min_{x\in\partial\B(1)} \Phi_x(\alpha),\\
 M_\alpha &= \max_{x\in\partial\B(1)} \Phi_x(\alpha)
\end{align}
the minimal and maximal distances from 
a boundary point of~$\B(1)$ to the boundary of 
the cell~$\cl_0(\alpha)$, see Figure~\ref{f_central_cell}.
\begin{figure}
\begin{center}
\includegraphics[width=0.64\textwidth]{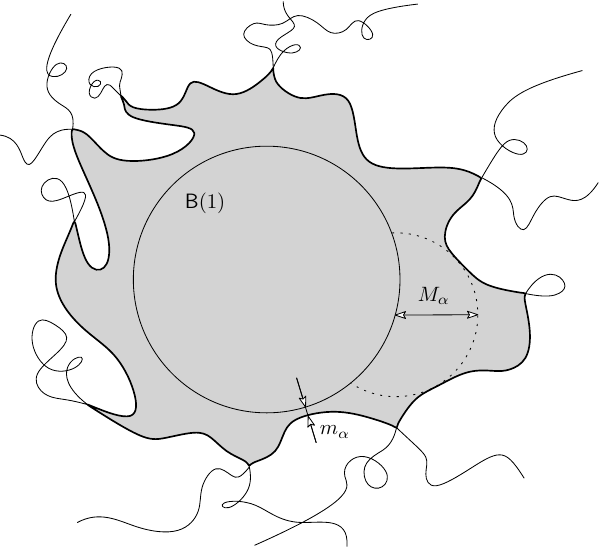}
\caption{The central cell (in this case,
formed by five bi-infinite trajectories).}
\label{f_central_cell}
\end{center}
\end{figure}
Now, \eqref{rho_exponential} immediately implies
that $2\alpha\ln(1+m_\alpha)$ is an Exp(1)
random variable
(so, informally, $m_\alpha$ is of order~$\alpha^{-1}$
with a random factor in front); 
also (the second part of)
Theorem~2.20 of~\cite{CP20} implies that,
for \emph{fixed} $x\in\partial\B(1)$,
$\alpha \big(\Phi_x(\alpha)\big)^2$
is \emph{approximately} an Exp(1)
random variable
(so, informally, the distance from a generic
boundary point of~$\B(1)$ to the boundary
of the cell is of order $\alpha^{-1/2}$;
again, with a random factor in front).
We now obtain that~$M_\alpha$ is concentrated
around $\sqrt{\frac{\ln \alpha}{2\alpha}}$
(without any random factors), and, moreover, even the a.s. convergence takes place
(here we, of course, assume that,
as mentioned earlier,
the random interlacement
process is constructed for all $\alpha\geq 0$
simultaneously, in such a way that $\BRI(\alpha_1)$
is dominated by $\BRI(\alpha_2)$ for $\alpha_1<\alpha_2$,
meaning that $M_\alpha$ is nonincreasing):
\begin{theo}
\label{t_central_cell}
 It holds that 
\begin{equation}
\label{eq_M_central_cell}
\Big(\frac{\ln \alpha}{2\alpha}\Big)^{-1/2} M_\alpha 
 \to 1 \text{ a.s., as }\alpha\to\infty.
\end{equation}
\end{theo}
Let us also remark that it would be interesting
to investigate the shape of the
(suitably rescaled) ``interface''
(that is, the boundary of $\cl_0(\alpha)$)
seen in some window
around a typical boundary point (say, $1$)
of~$\B(1)$ in the regime $\alpha\to\infty$. 

The rest of the paper will be organized in the following way: in Section~\ref{s_aux}, we recall some facts
and prove a few technical lemmas about the
conditioned Brownian motion, notably,
in Section~\ref{s_further_cond_BM} we introduce
and discuss a new martingale for the conditioned
Brownian motion.
Then, 
in Section~\ref{s_proofs} we give the proofs
of Theorems~\ref{t_conv_Poisson}, \ref{t_amoeba},
and~\ref{t_central_cell}.

\section{Some auxiliary facts}
\label{s_aux}
First, we recall a basic fact about hitting
circles centered at the origin by the conditioned
Brownian motion.
Since $\tfrac{1}{\ln|\hW_t|}$
is a local martingale, the optional stopping theorem implies
that for any $1<a<|x|<b<\infty$
\begin{equation}
\label{hitting_condBM}
 \IP_x[\htau(b)<\htau(a)] = \frac{(\ln a)^{-1}-(\ln |x|)^{-1}}
 {(\ln a)^{-1}-(\ln b)^{-1}}
   =\frac{\ln(|x|/a)  \times \ln b}{\ln (b/a) \times\ln |x|}.
\end{equation}
Sending~$b$ to infinity in~\eqref{hitting_condBM}
we also obtain that for $1\leq a\leq |x|$
\begin{equation}
\label{escape_condBM}
 \IP_x[\htau(a)=\infty] = 1-\frac{\ln a}{\ln |x|}.
\end{equation}

Also, in the following,
we will need to refer to the fact that
\begin{equation}
\label{diff_logs}
 \ln|x| = \ln|y| 
 + \ln\big(1+\tfrac{|x|-|y|}{|y|}\big)
  = \ln|y| + O\big(\tfrac{|x-y|}{|y|}\big).
\end{equation}

We denote by $\hW^{x,r}$ the Brownian motion
conditioned on not hitting~$\B(x,r)$;
it is straightforward to check that it can be obtained
from the Brownian motion conditioned by ``canonical'' conditioned Brownian motion~$\hW$ via the linear space-time transformation,
\[
  \hW^{x,r}_t = x+\hW_{r^2 t}.
\]

Then, we need a fact about
the conditional entrance measure of a Brownian motion
(we need to state and prove it here because
the second statement of Lemma~3.3 of~\cite{CP20}
is incorrect, 
 a logarithmic factor is missing there):
\begin{lem}
\label{l_cond_entrance_m}
 Let $0<2r<|x|<R$, and let~$\nu_x^{r,R}$ be 
 the conditional
entrance measure of the Brownian motion started at~$x$ to~$\B(r)$, 
given that $\tau(r)<\tau(R)$. 
Then, we have (abbreviating $s:=|x|$)
\begin{equation}
\label{eq_cond_entrance_m}
 \Big|\frac{d\nu_x^{r,R}}{d\hm_{\B(r)}}-1\Big| 
 = O\Big(\frac{r\ln\frac{R}{r}}{s\ln(1+\frac{R}{s})}\Big)
\end{equation}
(here, $\hm_{\B(r)}$ is the harmonic measure on~$\B(r)$,
which is also uniform on~$\partial\B(r)$).
\end{lem}

\begin{proof}
 Without restricting generality, one can assume that
$R-s\geq s/2$ (otherwise, one can just condition the
first entry point to~$\partial\B\big(\tfrac{2}{3}s\big)$).
Denote also by $\hat \nu$ the conditional
entrance measure of the Brownian motion 
started at~$x$ to~$\partial\B(R)$, 
given that $\tau(R)<\tau(r)$. 
Consider $M\subset \partial\B(r)$ and note first
that, for~$y$ with $|y|\geq 2r$
\begin{equation}
\label{uncond_entrance_m}
  \IP_y\big[W_{\tau(r)}\in M\big]
   = \hm_{\B(r)}(M)\big(1+O\big(\tfrac{r}{|y|}\big)
   \big);
\end{equation}
this is a well-known fact that can be obtained
 from the explicit formula for the (unconditional)
entrance measure to a disk from outside; see, e.g.\ 
Theorem~3.44 of~\cite{MP10} or~\eqref{Poisson_kernel}
below.
Then, we write
\begin{align*}
\lefteqn{
 \IP_x\big[W_{\tau(r)}\in M, \tau(r)<\tau(R)\big]
 }\\
   &= \IP_x\big[W_{\tau(r)}\in M\big]
    - \IP_x\big[W_{\tau(r)}\in M, \tau(R)<\tau(r)\big]
    \\
    \intertext{\quad
\footnotesize (by \eqref{uncond_entrance_m} and conditioning on the 
 entrance point to $\partial\B(R)$)}
 &=\hm_{\B(r)}(M)\big(1+O\big(\tfrac{r}{s}\big)\big)
    - \IP_x[\tau(R)<\tau(r)]
       \int_{\partial\B(R)} \IP_z\big[W_{\tau(r)}\in M\big]
         \, d{\hat\nu}(z)\\
    \intertext{\quad
\footnotesize (again by \eqref{uncond_entrance_m})}    
&= \hm_{\B(r)}(M)\big(1+O\big(\tfrac{r}{s}\big)\big)
    -  \IP_x[\tau(R)<\tau(r)]
     \hm_{\B(r)}(M)\big(1+O\big(\tfrac{r}{R}\big)\big)\\
 &= \hm_{\B(r)}(M)
  \big(\IP_x[\tau(r)<\tau(R)]+O\big(\tfrac{r}{s}\big)\big).
\end{align*}
We then divide the above by
$\IP_x[\tau(r)<\tau(R)] = \frac{\ln\frac{R}{s}}{\ln\frac{R}{r}}$
to arrive to~\eqref{eq_cond_entrance_m}
(recall the remark at the beginning of the proof).
\end{proof}

Next, we recall a fact about the capacity of a union
of two disks:
\begin{lem}
\label{l_cap_smalldisk}
  Assume that $r<|y|-1$. We have
\begin{equation} 
\label{eq_cap_smalldisk}
\hcapa\big( \B(y,r)\big)=
 \capa\big(\B(1)\cup \B(y,r)\big)  = \frac{2}{\pi}\cdot
 \frac{\ln^2|y|+O\big(\frac{r(1+|\ln r|+\ln|y|)\ln|y|}{|y|}\big)}
{\ln r^{-1}+\ell_y+\ln|y|
 + O\big(\frac{r}{|y|-1}\ln\frac{|y|-1}{r} \big)}.
\end{equation}
\end{lem}
\begin{proof}
 It is a reformulation of Lemma~3.11~(iii) of~\cite{CP20}.
\end{proof}

In the following subsections, we collect some
``more advanced'' auxiliary results about the
conditioned Brownian motion.

\subsection{Conditioned Brownian motion: martingales
and hitting probabilities}
\label{s_further_cond_BM}
Now, we will need some facts about hitting probabilities
for conditioned Brownian motions; however, to be able to 
estimate these probabilities,
we first develop a method which is ``cleaner'' than
the one used, e.g.\ in Lemma~3.7 of~\cite{CP20}.
\begin{lem}
\label{l_mart_hW}
 Assume that, for some open set $\Lambda\subset\R^2$
 such that $\Lambda\cap\B(1)=\emptyset$,
a function $g:\Lambda \mapsto\R$ is harmonic.
Then $g(\hW)/\ln|\hW|$ is a local martingale.
\end{lem}
\begin{proof}
Let~$x\in\Lambda$ and consider any bounded closed subset~$G$
of~$\Lambda$ that contains~$x$ in its interior. 
Let $\tau$ be the hitting time of~$\Lambda\setminus G$
and let~$t>0$.
Now, recall~\eqref{df_hat_p} and write
(note that none of the trajectories up to time $t\wedge \tau$
 can touch the unit disk and therefore be killed)
\[
\IE_x \frac{g(\hW_{t\wedge \tau})}{\ln |\hW_{t\wedge \tau}|}    
 = \frac{1}{\ln|x|} \IE_x g(W_{t\wedge \tau})
 = \frac{g(x)}{\ln|x|}
\]
since $g(W)$ is a local martingale, so (since $\hW$ is also Markovian) we obtain that $\frac{g(\hW)}{\ln|\hW|}$ is a local martingale.
\end{proof}

Next, we know
(see e.g.\ Theorem~3.44 of~\cite{MP10})
 that for $x\notin \B(1)$ and~$z\in \partial\B(1)$ 
\begin{equation}
\label{Poisson_kernel}
 H(x,z) = \frac{|x|^2-1}{2\pi|z-x|^2}
\end{equation}
is the Poisson kernel on $\R^2\setminus \B(1)$,
i.e., the density of the entrance measure 
to~$\B(1)$ when starting at~$x$.
Define for $x,y\notin\B(1)$, $x\neq y$,
\begin{align}
 L(x,y) &=  \frac{1}{\ln|x|}
 \bigg(\ln|x|-\ln|x-y|+\int_{\partial\B(1)}\ln|z-y|
 H(x,z)\, dz\bigg)
 \label{cont_mod_GF} \\
&= 1 +  \frac{1}{\ln|x|}
 \bigg(-\ln|x-y|
 +\frac{|x|^2-1}{2\pi} \int_{\partial\B(1)}
 \frac{\ln|z-y|}{|z-x|^2} 
 \, dz\bigg),
\label{cont_mod_GF_explicit}
\end{align}
Writing $x=x_0(1+\delta)$ with $x_0\in\partial \B(1)$ and $\delta>0$, this definition can be rewritten as 
\begin{align}
L(x,y) 
& = 1 + \frac{1}{\ln|x|}
\bigg( -\ln\frac{|x-y|}{|x_0-y|}
 + \frac{|x|^2-1}{2\pi} \int_{\partial\B(1)} 
 \frac{\ln\frac{|z-y|}{|x_0-y|}} {|z-x|^2}\, dz \bigg) 
 \nonumber  \\
& = 1 + \frac{1}{\ln|x|}
\Bigg( -\ln\frac{|x-y|}{|x_0-y|}
 + \frac{|x|^2-1}{2\pi} \int_0^\pi
 \frac{\ln\frac{|z^{(\phi)}-y||z^{(-\phi)}-y|}
 {|x_0-y|^2}}
  {|z^{(\phi)}-x|^2}\, d\phi \Bigg), 
\label{integral_0_pi}
\end{align}
where 
$z^{(\phi)}=x_0 e^{i\phi}$.
To cover the case $\delta=0$, 
define for $x_0\in \partial \B(1), y\notin \B(1)$ (denoting by $(\cdot, \cdot)$ the scalar product in $\R^2$),
\begin{align}
 L(x_0,y) &= \frac{(y-x_0,y)}{|x_0-y|^2}
  + \frac{1}{\pi} 
 \int_0^\pi
 \frac{\ln\frac{|z^{(\phi)}-y||z^{(-\phi)}-y|}{|x_0-y|^2}}
  {|z^{(\phi)}-x_0|^2}\, d\phi.
\label{cont_mod_GF_explicit_circle}
\end{align}

Then, partly as a consequence of Lemma~\ref{l_mart_hW}, 
we have
\begin{prop}
\label{p_L_properties}
 For any fixed~$y\notin\B(1)$, it holds that
\begin{itemize}
 \item[(i)] the function 
 $\big(\R^2\setminus(\B(1)\cup\{y\})\big)
 \cup\partial\B(1) \to \R, x\mapsto L(x, y)$ is continuous; 
 \item[(ii)] the process $L(\hW_t,y)$ is a local martingale;
 \item[(iii)] $L(x,y)\to 0$ as $x\to \infty$;
 \item[(iv)] for any fixed~$r>0$, $L(x,y)$
 is uniformly bounded in $x\in\R^2\setminus(\B(1)\cup\B(y,r))$;
 \item[(v)] $L(x,y)>0$ for all 
 $x\in (\R^2\setminus\B(1))\cup\partial\B(1)$, $x\neq y$.
\end{itemize}
\end{prop}
\begin{proof} 
We fix $y\notin \B(1)$ once and for all. 
We begin with the proof of (i). 
The continuity outside of $\B(1)$ is immediate. Let us prove that for all $x_0\in\partial\B(1)$,
\begin{equation}
\label{L_boundary}
 \lim_{x\to x_0} L(x,y) 
 = \frac{(y-x_0,y)}{|x_0-y|^2}
  + \frac{1}{\pi} 
 \int_0^\pi
 \frac{\ln\frac{|z^{(\phi)}-y||z^{(-\phi)}-y|}{|x_0-y|^2}}
  {|z^{(\phi)}-x_0|^2}\, d\phi.
\end{equation}
To prove that, we assume that $x=(1+\delta)x_0$
and then take the limit $\delta\downarrow 0$
with the help of~\eqref{integral_0_pi};
we will prove that convergence uniformly in~$x_0$, together with the continuity in $x_0$ of $L(x_0, y)$, and \eqref{L_boundary} will follow by elementary arguments. We proceed term by term.
First,
we compute that for fixed $x_0\in \partial \B(1)$, we have 
\begin{equation}
\ln |(1+\delta) x_0| = \ln (1+\delta) =\delta +O(\delta^2)
\end{equation}
and
\begin{align}
 \ln\frac{|(1+\delta)x_0-y|}{|x_0-y|}
& = \frac12 \ln\frac{|x_0-y + \delta x_0|^2}{|x_0-y|^2}
\nonumber\\
 & = \frac12 \ln\Big(1+ 2 \delta \cdot\frac{(x_0-y,x_0)}{|x_0-y|^2} + \delta^2 \cdot \frac{1}{|x_0-y|^2} \Big)
 \label{ln_x_x0ln}\\
 &= \delta \cdot 
  \frac{(x_0-y,x_0)}{|x_0-y|^2} + O(\delta^2).
 \label{ln_x_x0}
\end{align}
We derive that, for any $x_0\in\partial \B(1)$, when $x=x_0(1+\delta)$, we have 
\begin{equation}
\label{limln_x_x0}
1-\lim_{\delta\to 0} \frac{1}{\ln |x|} 
\ln\frac{|x-y|}{|x_0-y|}  
= 1 - \frac{(x_0-y, x_0)}{|x_0-y|^2}=\frac{(y-x_0,y)}{|x_0-y|^2}.
\end{equation}
We argue that this convergence is in fact 
uniform in~$x_0$. We observe that the coefficients
$\frac{(x_0-y,x_0)}{|x_0-y|^2}$ and $\frac{1}{|x_0-y|^2}$ appearing in~\eqref{ln_x_x0ln} are bounded, 
respectively by $\frac{|y|+1}{(|y|-1)^2}$ 
and $ \frac1{(|y|-1)^2}$. Using the fact that 
$|\ln(1+u)-u|\le 2 u^2$ for $u\in[-\frac12, \frac12]$, 
we see that the~$O(\delta^2)$ in~\eqref{ln_x_x0} is uniform in~$x_0$, 
and thus, so is the convergence in~\eqref{limln_x_x0}. 

We also compute $\lim_{\delta\to 0} \frac{|x|^2-1}{\ln |x|} = 2$ and see that this limit is uniform 
in~$x_0$ (in fact, $\frac{|x|^2-1}{\ln |x|}$ is independent of~$x_0$).

Finally, let us argue using the dominated convergence 
theorem that, still with $x=(1+\delta)x_0$,
\begin{equation}
\label{eq:convint}
  \int_0^\pi
\frac{\ln\frac{|z^{(\phi)}-y||z^{(-\phi)}-y|}{|x_0-y|^2}}
  {|z^{(\phi)}-x|^2}\, d\phi
  \underset{\delta\to 0}{\longrightarrow}\int_0^\pi
\frac{\ln\frac{|z^{(\phi)}-y||z^{(-\phi)}-y|}{|x_0-y|^2}}
  {|z^{(\phi)}-x_0|^2}\, d\phi,
\end{equation}
uniformly in~$x_0$.
For the moment, let us reason with fixed~$x_0$. 
It is clear that the integrand in the left-hand side 
of~\eqref{eq:convint} converges pointwise towards the integrand in the right-hand side of~\eqref{eq:convint}. Observing that $|x_0-z^{(\phi)}|\le |x-z^{(\phi)}|$, we see that the integrand in the right-hand side
of~\eqref{eq:convint} dominates the integrand in the left-hand side of~\eqref{eq:convint}. Let us check that the integrand in the right-hand side 
of~\eqref{eq:convint} is integrable; 
the only difficulty is around~$0$.
On the one hand, as $\phi \to 0$:
\begin{align}
\frac{|z^{(\phi)}-y|^2}{|x_0-y|^2} 
& = \frac{|x_0-y+x_0(e^{i \phi}-1)|^2}{|x_0-y|^2} \nonumber \\
& = 1 + 2\frac{(x_0-y, x_0 (e^{i \phi}-1))}{|x_0-y|^2} 
+ \frac{|e^{i \phi}-1|^2}{|x_0-y|^2}
\nonumber\\
& = 1 + 2 \phi \frac{(x_0-y, x_0 i)}{|x_0-y|^2} + 2\frac{(x_0-y, x_0 (e^{i \phi}-1-i\phi))}{|x_0-y|^2} + \frac{|e^{i \phi}-1|^2}{|x_0-y|^2} 
\label{eqlnratioexact}\\
& = 1 + 2 \phi \frac{(x_0-y, x_0 i)}{|x_0-y|^2}
+O(\phi^2),
\label{eqlnratioO}
\end{align}
so
\begin{align}
\ln \frac{|z^{(\phi)}-y||z^{(-\phi)}-y|}{|x_0-y|^2} 
&= \frac12 \ln\bigg(\frac{|z^{(\phi)}-y|^2}{|x_0-y|^2}\times \frac{|z^{(-\phi)}-y|^2}{|x_0-y|^2}\bigg) 
\nonumber\\
&= \frac12 \ln ( 1 + O(\phi^2))=O(\phi^2).
\label{eq:lnquO}
\end{align}
On the other hand, 
$|z^{(\phi)}-x_0|^2=|e^{i\phi}-1|^2=\phi^2 + O(\phi^3)$.
We obtain that the integrand on the right-hand side 
of~\eqref{eq:convint} is bounded around~$0$, 
thus, it is integrable. 
The dominated convergence theorem yields the convergence in~\eqref{eq:convint}, for fixed~$x_0$. 

To conclude, we need to show that this convergence is uniform in~$x_0$. We are simply going to refine the above arguments by showing that our estimates are uniform.
Let us denote
\begin{align}
f(\phi, \delta)
& :=  \sup_{x_0\in\partial \B(1)} 
\Bigg| 
 \frac {\ln\frac{|z^{(\phi)}-y||z^{(-\phi)}-y|}{|x_0-y|^2}}
       {|z^{(\phi)}-x_0|^2}
 -
\frac {\ln\frac{|z^{(\phi)}-y||z^{(-\phi)}-y|}{|x_0-y|^2}}
       {|z^{(\phi)}-x|^2}
 \Bigg|\nonumber\\
& \ = \Big(1-\frac{|e^{i \phi}-1|^2}{|e^{i\phi}-(1+\delta)|^2}\Big) \times g(\phi) \label{deff(phi, delta)},
\end{align}
with 
\begin{equation}\label{defg(phi)}
g(\phi)
:= 
\frac {\sup_{x_0\in\partial \B(1)} \big|{\ln\frac{|z^{(\phi)}-y||z^{(-\phi)}-y|}{|x_0-y|^2}}\big|}
      {|e^{i\phi}-1|^2}.
\end{equation}

Let us show that~$g(\phi)$ is bounded: the numerator on the right-hand side of \eqref{defg(phi)} is clearly bounded, so we only have to check that~$g$ does not diverge at~$0$. This comes from~\eqref{eqlnratioexact}, \eqref{eqlnratioO} and~\eqref{eq:lnquO}: observe that the coefficients in~\eqref{eqlnratioexact} are bounded 
in~$x_0$. Using the fact that $|\ln(1+u)-u|\le 2 u^2$ for $u\in[-\frac12, \frac12]$, it is
straightforward to establish that the~$O(\phi^2)$ in~\eqref{eq:lnquO} is uniform in~$x_0$, 
hence~$g$ is bounded around~$0$.

In view of~\eqref{deff(phi, delta)}, we see 
that $f(\cdot, \delta)$ is dominated by~$g$ and converges pointwise towards~$0$. The dominated convergence theorem then yields
$\int_0^\pi f(\phi, \delta)d \phi \underset{\delta\to 0}{\longrightarrow} 0$. 
The uniform convergence in~\eqref{eq:convint} follows by exchanging the integration with the supremum and the absolute value.

From the boundedness of~$g$, we also derive that the right-hand side of~\eqref{eq:convint}, hence also $L(x_0, y)$, is continuous in~$x_0$, as desired.
This concludes the proof of item~(i).

Next, item~(ii)
follows from Lemma~\ref{l_mart_hW} (recall that
the Poisson kernel is harmonic with respect to
its first variable). Items~(iii) and~(iv) are straightforward to obtain.

Finally, item~(v) follows from the optional stopping theorem:
a direct computation implies that $L(z,y)>1$ for all $z\in \partial\B(y,r)$ with small enough~$r$. Then,
the optional stopping theorem together 
with~(ii) and~(iii) implies that
\[
L(x,y) = \IP_x[\htau(y,r)<\infty] 
  \IE_x\big(L(\hW_{\htau(y,r)},y)
     \mid \htau(y,r)<\infty\big)>0.
\]
This concludes the proof of
Proposition~\ref{p_L_properties}.
\end{proof}

In the next result, we
are interested in the regime when the radius of the disk (to be hit)
is small, and the starting point is also close to the disk.
\begin{lem}
\label{l_cond_escape_or_hit}
Consider some fixed $\delta_0>0$. For every $x,y \notin B(1)$ such that 
$|x-y|<\frac{1}{2}$
and
$|x-y|\leq (|y|-1)^{1+\delta_0}$, we have, uniformly for all $r \in (0, |x-y|)$:
\begin{align}
\IP_x\big[\htau(y,r)<\infty\big] &=  \frac{\ln|x-y|^{-1}
 +\ell_y +\ln|y|}{\ln r^{-1} +\ell_y +\ln|y|}
 \Big(1+O\big(\tfrac{|x-y|\tell_y}{(|y|-1)
  (\ln|x-y|^{-1}+\ln|y|)}\big)\Big),
  \label{eq_cond_hit}\\
\IP_x\big[\htau(y,r)=\infty\big] &=   
\frac{\ln \tfrac{|x-y|}{r}}{\ln r^{-1} +\ell_y +\ln|y|}
 + O\big(\tfrac{|x-y|\tell_y}{(|y|-1)
  (\ln|x-y|^{-1}+\ln|y|)}\big).
 \label{eq_cond_escape}
\end{align}
\end{lem}

\begin{proof}
Instead of relying on Lemma~3.7~(iii) of~\cite{CP20},
we will use the optional stopping theorem with the martingale
provided by Proposition~\ref{p_L_properties} and the stopping
time~$\htau(y,r)$.

We first establish some estimates.
Let us show that according to our assumptions, $\frac{|x-y|}{|y|-1}$ is bounded by 1 and $\frac{|y|-1}{|x|-1}$ is bounded by a constant depending on $\delta_0$. For the first quantity, observe that when $|y|\le 2$, we have $|x-y| \le (|y|-1)^{1+\delta_0} \le |y|-1$, while $|y|>2$ we simply have $\frac{|x-y|}{|y|-1} < \frac{1}{2}$.
For the second quantity, write when $|y|\le 1.7$, 
\[|x|-1 \ge |y|-1 -|x-y| \ge |y|-1 - (|y|-1)^{1+\delta_0}\ge (|y|-1) (1-0.7^{\delta_0}),\]
and note that $|y|> 1.7$ we have $|x|-1\ge |y|-|x-y|-1 > |y|-3/2$ which yields $\frac{|y|-1}{|x|-1} \le \frac{|y|-1}{|y|-\frac32}\le \frac72$.

 From these facts, we now deduce that, under the 
conditions of the lemma,
\begin{equation}
\label{compare_Hxzy}
H(x,z)=H(y,z)\big(1+O\big(\tfrac{|x-y|}{|y|-1}\big)\big)
\end{equation}
uniformly in~$z\in\partial\B(1)$.
Indeed, using~\eqref{Poisson_kernel}, we 
write
\begin{align*}
  \frac{H(x,z)}{H(y,z)}
 & = \frac{|x|^2-1}{|y|^2-1}\cdot 
  \frac{|y-z|^2}{|x-z|^2}\\
 & = 
 \Big(1+\frac{|x-y|^2+2(y,x-y)}{|y|^2-1}\Big)
 \Big(1+\frac{|x-y|^2 + 2(x-y,x-z)}{|x-z|^2}\Big)\\
 & =: (1+T_1)(1+T_2).
\end{align*}
Note that, since $|y|\geq 1$,
 we have $|y|^2-1\ge |y|(|y|-1)$,
 and recall that $\frac{|x-y|}{|y|-1}$ is bounded;
this shows that~$T_1$  
is~$O\big(\tfrac{|x-y|}{|y|-1}\big)$.
Then, writing $ \frac{|x-y|}{|x-z|} \le \frac{|x-y|}{|x|-1} = \frac{|x-y|}{|y|-1} \frac{|y|-1}{|x|-1}$, recalling that both fractions are bounded, we also find that $T_2$ is~$O\big(\tfrac{|x-y|}{|y|-1}\big)$.
By the boundedness of $\tfrac{|x-y|}{|y|-1}$, this implies~\eqref{compare_Hxzy}.

We now consider $r\in(0, |x-y|)$. 
Notice that our assumptions imply $r<|y|-1$,
so that $\B(y,r)\cap\B(1)=\emptyset$. Indeed, if $|y|\le 2$, then $r< |x-y| \le (|y|-1)^{1+\delta_0} \le |y|-1$, while if $|y|>2$ then $|y|-1 > 1 >r$.

Recalling~\eqref{behaviour_ell_x} and using~\eqref{diff_logs}
(to substitute $\ln |x|$ or $\ln|u|$
by $\ln|y|$),
we notice that~\eqref{compare_Hxzy} 
permits us to obtain 
\begin{align*}
L(x,y) &=  \frac{\ln |x-y|^{-1}
   + \ell_y + \ln|y| 
   + O\big(\tfrac{|x-y|\tell_y}{|y|-1}\big)}
 {\ln|y|+O\big(\tfrac{|x-y|}{|y|}\big)}
 \intertext{and, for any $u\in\partial\B(y,r)$}
L(u,y) &=  \frac{\ln r^{-1}
   + \ell_y + \ln|y| 
   + O\big(\tfrac{r\tell_y}{|y|-1}\big)}
 {\ln|y|+O\big(\tfrac{r}{|y|}\big)} .
\end{align*}
Now, by the assumption that $|x-y|\le (|y|-1)^{1+\delta_0}$, and lower bounding $\ell_y$ by $\ln(|y|-1)$, we have 
\[\frac{1}{1+\delta_0} \ln|x-y|^{-1} + \ell_y \ge 0 \]
hence $\ln|x-y|^{-1} + \ell_y + \ln|y| \ge \frac{\delta_0}{1+\delta_0}(\ln|x-y|^{-1}+\ln|y|)$.

We are now ready to conclude. We obtain from the optional stopping theorem
(recall 
Proposition~\ref{p_L_properties}~(iii) and (iv))
\begin{align*}
 \IP_x\big[\htau(y,r)<\infty\big]
 &= \frac{
  \big(\ln |x-y|^{-1}
   + \ell_y + \ln|y| 
   + O\big(\tfrac{|x-y|\tell_y}{|y|-1}\big)\big)
   \big(\ln|y|+O\big(\tfrac{r}{|y|}\big)\big)
   }
 {
 \big(\ln r^{-1}
   + \ell_y + \ln|y| 
   + O\big(\tfrac{r\tell_y}{|y|-1}\big)\big)
   \big(\ln|y|+O\big(\tfrac{|x-y|}{|y|}\big)\big)
   }\\
 &= \frac{\ln|x-y|^{-1}
 +\ell_y +\ln|y|}{\ln r^{-1} +\ell_y +\ln|y|}
 \Big(1+O\big(\tfrac{|x-y|\tell_y}{(|y|-1)
  (\ln|x-y|^{-1}+\ln|y|)}\big)\Big),
\end{align*}
thus proving~\eqref{eq_cond_hit}.
Then, \eqref{eq_cond_escape}
is a direct consequence of~\eqref{eq_cond_hit}.
\end{proof}

\subsection{Traces of different conditioned
Brownian motions}
\label{s_comparing_traces}
Next, the goal is to be able to compare traces left by the Brownian 
motion conditioned on not hitting a specific (typically small) disk
and the usual conditioned (on not hitting~$\B(1)$) Brownian motion
additionally conditioned on not hitting that disk. Recall that the Brownian motion conditioned on not hitting $B(y,r)$ is denoted by $\hW^{y,r}$.
\begin{lem}
\label{l_cond_B(y,r)}
Assume that $0<r<r'$ and let $y$ be such that
$\B(y,r')\cap \B(1)=\emptyset$
(meaning that $s:=|y|-r'-1$ is strictly positive).
Let us denote also by~$\tau^*(z,r')$ the hitting time
of~$\partial\B(z,r')$ by~$\hW^{y,r}$. Then, 
for any $x\in \big(\B(y,r')\setminus \B(y,r)\big)
\cup \partial\B(y,r)$
\begin{equation}
\label{eq_cond_B(y,r)} 
\Big|\frac{d\IP_x\big[\hW_{[0,\htau(y,r')]}
\in\cdot\,\mid \htau(y,r')<\htau(y,r)\big]}
{d\IP_x\big[\hW^{y,r}_{[0,\tau^*(y,r')]}\in\cdot\,\big]}-1\Big| 
= O\big(\tfrac{r'}{s\ln s}\big). 
\end{equation}
\end{lem}
\begin{proof}
First, we note that,
analogously to Section~2.1 of~\cite{CP20},
it is possible to define the diffusion~$\hW$
conditioned on not touching $\B(y,r)$ anymore
even for a starting point on $\partial\B(y,r)$.
 Since the estimates we obtain below
 (in the case $x\in \B(y,r')\setminus \B(y,r)$)
will be uniform in~$x$, it is enough to 
prove the result for $x\in \B(y,r')\setminus \B(y,r)$.

Let $\Gamma$ be a set of (finite) trajectories,
 having the following property:
a trajectory belonging to this set has to start at~$x$, 
it cannot touch~$\B(y,r)$, and it ends on its first visit 
to~$\partial\B(y,r')$.
 Then, we have (by an obvious adaptation of Lemma~2.1
of~\cite{CP20})
\begin{align*}
 \IP_x\big[\hW^{y,r}_{[0,\tau^*(y,r')]}\in\Gamma\,\big]
 & = \IP_x\big[W_{[0,\tau(y,r')]}
\in\Gamma\,\mid \tau(y,r')<\tau(y,r)\big]\\
&= \frac{\IP_x\big[W_{[0,\tau(y,r')]}
\in\Gamma\, \big]}
{\IP_x\big[\tau(y,r')<\tau(y,r)\big]},
\end{align*}
and
\begin{align*}
 \IP_x\big[\hW_{[0,\htau(y,r')]}
\in\Gamma\,\mid \htau(y,r')<\htau(y,r)\big]
&= \frac{\IP_x\big[\hW_{[0,\htau(y,r')]}
\in\Gamma\, \big]}
{\IP_x\big[\htau(y,r')<\htau(y,r)\big]}\\
\intertext{\qquad 
\qquad \qquad \qquad 
\qquad \qquad \qquad 
\qquad 
\footnotesize (by Lemma~3.6 of~\cite{CP20})}
&=  \frac{\IP_x\big[W_{[0,\tau(y,r')]}
\in\Gamma\, \big]}{\IP_x\big[\tau(y,r')<\tau(y,r)\big]}
\big(1+O\big(\tfrac{r'}{s\ln s}\big)\big),
\end{align*} 
which implies~\eqref{eq_cond_B(y,r)}.
\end{proof}

The above result compares the traces left in $\B(y,r')$ by~$\hW^{y,r}$
and (conditioned)~$\hW$ before going out of $\B(y,r')$ 
for the first time;
it is important to observe that 
it allows us to couple these traces with probability
close to~$1$ (when the right-hand side 
of~\eqref{eq_cond_B(y,r)} is small).
With $r':=r \ln^\theta r^{-1}$,
we also need to compare 
the full traces left in~$\B(y,r')$
by these two processes:
\begin{lem}
\label{l_cond_full_B(y,r)}
Fix $\theta\geq 2$ and $\delta_0>0$.
Let~$r<1$ and $y\notin\B(1)$ be such that 
\begin{equation}
\label{r_much_smaller_y-1}
 r\ln^{2\theta} r^{-1}
 < \min\big(\tfrac{1}{2},|y|-1, (|y|-1)^{1+\delta_0}\big).
\end{equation}
Then, for any 
$x\in \big(\B(y,\ln^\theta r^{-1})\setminus\B(y,r)\big)
\cup \partial\B(y,r)$, 
we have
\begin{equation}
\label{eq_cond_full_B(x,r)} 
\Big|\frac{d\IP_x\big[\hW_{[0,\infty]}\cap
\B(y,r\ln^\theta r^{-1})
\in\cdot\,\mid \htau(y,r)=\infty\big]}
{d\IP_x\big[\hW^{y,r}_{[0,\infty]}\cap \B(y,r\ln^\theta r^{-1})
 \in\cdot\,\big]}-1\Big| 
= O\big(\tfrac{\ln\ln r^{-1}}{\ln r^{-1}+\ln|y|}
+ \tfrac{r\tell_y\ln^{2\theta}r^{-1}}{|y|-1}\big). 
\end{equation}
\end{lem}
\begin{proof}
Once again, as in the proof 
of Lemma~\ref{l_cond_B(y,r)}, it is enough
to obtain the proof for the case
$x\in\B(y,\ln^\theta r^{-1})\setminus\B(y,r)$.
 Abbreviate $r_1:= r\ln^\theta r^{-1}$, 
 $r_2:= r\ln^{2\theta} r^{-1}$.
The idea of the proof is the following: first, by Lemma~\ref{l_cond_B(y,r)},
the traces left by the two processes on $\B(y,r_2)$
(and therefore on~$\B(y,r_1)$) up to their first hitting time 
of~$\partial\B(y,r_2)$ have almost the same law. 
Then, both processes may return a few times from~$\partial\B(y,r_2)$
to~$\partial\B(y,r_1)$, and we argue that the number of such returns
has almost the same law (in fact, it is precisely $\text{Geom}_0(1/2)$
for $\hW^{y,r}$); we also argue that the entrance points of these
returns have almost the same law. Then, each pair of additional
excursions from~$\partial\B(y,r_1)$ to~$\partial\B(y,r_2)$
again can be coupled with high probability by Lemma~\ref{l_cond_B(y,r)}.

Now, we fill in the details. First, note that 
(e.g., by~(2.17) of~\cite{CP20})
\begin{equation}
\label{escape_hW_y_r}
 \text{ for any }x\in\partial\B(y,r_2), \qquad 
 \IP_x[\tau^*(y,r_1)<\infty] = 
\frac{\ln ({r_1}/{r})}{\ln({r_2}/{r})}
  = \frac{1}{2}.
\end{equation}
For the (additionally) conditioned process~$\hW$,
we have, for $x\in\partial\B(y,r_2)$
\begin{align}
 \lefteqn{
 \IP_x[\htau(y,r_1)<\infty\mid \htau(y,r)=\infty]
 }
 \nonumber\\
 \intertext{
\footnotesize ($\nu$ being the conditional entrance measure
 to $\partial\B(y,r_1)$ from~$x$)}
 &= \frac{\IP_x[\htau(y,r_1)<\infty]\IP_\nu[\htau(y,r)=\infty]}
 {\IP_x[\htau(y,r)=\infty]} \nonumber\\
  \intertext{
\footnotesize (by Lemma~\ref{l_cond_escape_or_hit})}
& = \frac{\frac{\ln r_2^{-1}
  +\ell_y +\ln|y|}{\ln r_1^{-1} +\ell_y +\ln|y|}
  \big(1+O\big(\tfrac{r_2\tell_y}{(|y|-1)
  (\ln r_2^{-1}+\ln|y|)}\big)\big)
  \big(\frac{\ln\frac{r_1}{r}}{\ln r^{-1} +\ell_y +\ln|y|}
  +O\big(\tfrac{r_1\tell_y}{(|y|-1)
  (\ln r_1^{-1}+\ln|y|)}\big)\big)}
  {\frac{\ln\frac{r_2}{r}}{\ln r^{-1} +\ell_y +\ln|y|}
  +O\big(\tfrac{r_2\tell_y}{(|y|-1)
  (\ln r_2^{-1}+\ln|y|)}\big)}
\nonumber\\  
& = \frac{\frac{\ln r_2^{-1}
  +\ell_y +\ln|y|}{\ln r_1^{-1} +\ell_y +\ln|y|}\cdot
  \frac{\ln\frac{r_1}{r}}{\ln r^{-1} +\ell_y +\ln|y|}}
  {\frac{\ln\frac{r_2}{r}}{\ln r^{-1} +\ell_y +\ln|y|}}
  \Big(1+O\big(\tfrac{r\tell_y\ln^{2\theta}r^{-1}}
  {|y|-1}\big)\Big)
\nonumber\\
\intertext{
\footnotesize (recall that 
$
\frac{\ln ({r_1}/{r})}{\ln({r_2}/{r})}
  = \frac{1}{2}$)}
   &= \frac{1}{2}  
\Big(1+O\big(\tfrac{\ln\ln r^{-1}}{\ln r^{-1} + \ln|y|}
 + \tfrac{r\tell_y\ln^{2\theta}r^{-1}}{|y|-1}\big)\Big).
\label{prob_cond_hW_return} 
\end{align}
So, with~\eqref{escape_hW_y_r} 
and~\eqref{prob_cond_hW_return}, we are now able to couple
the excursion counts of the two processes with 
probability close to~$1$.

Next, we need to be able to couple the entrance points to~$\partial\B(y,r_1)$
with high probability. Let~$M$ be a (measurable) subset 
of~$\partial\B(y,r_1)$.
 First, let us note that Lemma~\ref{l_cond_entrance_m}
implies that
\begin{equation}
\label{entr_B_r1_cond_W}
 \IP_u[\hW^{y,r}_{\tau^*(y,r_1)}\in M\mid \tau^*(y,r_1)<\infty]
  = \frac{|M|}{2\pi r_1} 
  \big(1+O\big(\tfrac{1}{\ln^\theta r^{-1}}\big)\big)
\end{equation}
uniformly in $u\in\partial\B(y,r_2)$. Let us obtain an analogue 
of~\eqref{entr_B_r1_cond_W} for~$\hW$ conditioned on $\{\htau(y,r)=\infty\}$.
Lemma~\ref{l_cond_escape_or_hit} implies that
\begin{equation}
\label{escape_R_r_hW}
 \IP_u[\htau(y,r)>\htau(R)] = \frac{2\ln\ln r^{-1}}{\ln r^{-1}+\ell_y+\ln|y|}
 \Big(1+O\big(\tfrac{r\tell_y \ln^\theta r^{-1}}{|y|-1}\big)\Big) + o(1)
\end{equation}
as $R\to\infty$, uniformly in $u\in\partial\B(y,r_1)$.
Assume without loss of generality that $r_1<r_2/2$. 
For~$R$ such that $\B(y,r_2) \subset \B(R)$,
 abbreviate $ G_R = \big\{\tau(y,r_1)<\tau(R)<\tau(1)\big\}$.
Define the (possibly infinite) random variable
\[
 T_R = 
  \begin{cases}
   \inf\big\{t\geq 0: W_t\in\B(y,r_2/2),
   W_{[t,\tau(y,r_1)]}\cap \partial\B(y,r_2) = \emptyset\big\} &
    \text{ on } G_R,\\
    \infty & \text{ on } G_R^\complement
  \end{cases}
\]
to be the time when the last (before hitting~$\partial\B(y,r_1)$) 
excursion between
$\partial\B(y,r_2/2)$ and~$\partial\B(y,r_1)\cup\partial\B(y,r_2)$ starts.
Note that~$T_R$ is not a stopping time;
and
the law of the excursion
that begins at time~$T_R$ is the law of a Brownian excursion
conditioned to reach~$\partial\B(y,r_1)$
before going to~$\partial\B(y,r_2)$.
We denote that excursion by~$\tW$ (with its initial time
reset to~$0$) and denote by~$\sigma$ the time
it hits~$\partial\B(y,r_1)$.
Let~$\nu_R$ be the joint law of 
the pair $(T_R,W_{T_R})$ under~$\IP_x$.
Abbreviate also $\mathcal{H}=\R_+\times\partial\B(y,r_2/2)$
and observe that $\int_{\mathcal{H}}\,d\nu_R(t,y) = \IP_x[G_R]$.

We can then write for $x'\in\partial\B(y,r_2)$
\begin{align}
\lefteqn{
 \IP_{x'}\big[\hW_{\htau(y,r_1)}\in M, \htau(y,r_1)<\infty, \htau(y,r)=\infty\big]
}\nonumber\\
 &=\lim_{R\to\infty}
  \IP_{x'}\big[\hW_{\htau(y,r_1)}\in M, \htau(y,r_1)<\htau(R)<\htau(y,r)\big]
 \nonumber\\
   \intertext{~
\footnotesize (by Lemma~2.1 of~\cite{CP20})}
 &=\lim_{R\to\infty} 
   \IP_{x'}\big[W_{\tau(y,r_1)}\in M, \tau(y,r_1)<\tau(R)<\tau(y,r)
    \mid \tau(R)<\tau(1)\big]
 \nonumber\\
  &=\lim_{R\to\infty}\frac{\ln R}{\ln|x'|} 
   \IP_{x'}\big[W_{\tau(y,r_1)}\in M, \tau(y,r_1)<\tau(R)
     <\tau(y,r)\wedge\tau(1)\big]
 \nonumber\\
   &=\lim_{R\to\infty}\frac{\ln R}{\ln|x'|} 
   \int_{\mathcal{H}} d\nu_R(t,z)
   \int_{M}\IP_z[\tW_\sigma\in du]
   \IP_u[\tau(R)
     <\tau(y,r)\wedge\tau(1)]
 \nonumber\\     
    \intertext{~
\footnotesize (again by Lemma~2.1 of~\cite{CP20})}
    &=\lim_{R\to\infty}\frac{1}{\ln|x'|} 
   \int_{\mathcal{H}} d\nu_R(t,z)
   \int_{M}\IP_z[\tW_\sigma\in du]
   \IP_u[\htau(y,r)>\htau(R)] \ln|u|.
\label{calc_with_M}
\end{align}
On the other hand, in the same way, one can obtain
\begin{align}
\lefteqn{
 \IP_{x'}\big[\htau(y,r_1)<\infty, \htau(y,r)=\infty\big]
}\nonumber\\
    &=\lim_{R\to\infty}\frac{1}{\ln|x'|} 
   \int_{\mathcal{H}} d\nu_R(t,z)
   \int_{\partial \B(y,r_1)}\IP_z[\tW_\sigma\in du]
   \IP_u[\htau(y,r)>\htau(R)] \ln|u|.
\label{calc_with_partial_B(y,r_1)}
\end{align}
Then, \eqref{diff_logs} implies that 
$\frac{\ln|u|}{\ln|x'|}=\ln|y|+O(r_2/|y|)$
uniformly in $u\in\partial \B(y,r_1)$, and 
Lemma~\ref{l_cond_entrance_m} implies that 
\[
\IP_z[\tW_\sigma\in M]=\int_{M}\IP_z[\tW_\sigma\in du]
=\frac{|M|}{2\pi r_1}\big(1+O\big(\tfrac{1}{\ln^\theta r^{-1}}\big)\big)
\]
for any $z\in\partial\B(y,r_2/2)$.
Using also~\eqref{escape_R_r_hW}
together with~\eqref{calc_with_M}--\eqref{calc_with_partial_B(y,r_1)}, 
we then obtain that 
\begin{equation}
\label{entr_B_r1_cond_hW} 
 \IP_{x'}\big[\hW_{\htau(y,r_1)}\in M \mid \htau(y,r_1)<\infty,
 \htau(y,r)=\infty\big]=\frac{|M|}{2\pi r_1}
  \big(1+O\big(\tfrac{1}{\ln^\theta r^{-1}}
  +\tfrac{r\ln^\theta r^{-1}\tell_y}{|y|-1}\big)\big),
\end{equation}
which is the desired counterpart of~\eqref{entr_B_r1_cond_W}
and permits us to couple the starting points of the 
corresponding excursions with probability
at least $1-O\big(\tfrac{1}{\ln^\theta r^{-1}}
  +\tfrac{r\ln^\theta r^{-1}\tell_y}{|y|-1}\big)$.

Using the observation we made
after~\eqref{prob_cond_hW_return}, it is now
straightforward to conclude the proof of
Lemma~\ref{l_cond_full_B(y,r)}
(as outlined at the beginning of the proof).
\end{proof}

\subsection{Controlling the size of the Brownian amoeba}
\label{s_size_amoeba}
Next, we need a result that would permit us to control
the size of
the Brownian amoeba. 
For any $r\geq 1$, define the event
\begin{equation}
\label{df_event_cut_infty} 
 U_r = \big\{\B(1) \text{ is not connected to }\infty
\text{ in } \C\setminus \hW_{[\htau(r),\htau(2r)]} \big\}. 
\end{equation}
\begin{lem}
\label{l_go_around}
There exists $\gamma_0>0$ such that for all $x$ with~$|x|\geq 1$
we have
\begin{equation}
\label{eq_go_around}
\IP_x[U_{|x|}] \geq \gamma_0.
\end{equation}
\end{lem}

\begin{proof}
Let $r:=|x|$.
It is convenient to use Proposition~\ref{p_represent_hW}:
it is enough to consider the trace of the process
$\exp(Z_t+iB_t)$ (where $Z$ is Bes(3) and~$B$ is a
standard Brownian motion independent of~$Z$) before~$Z$
hits $\ln (2r)$. Denote
$\tau^Z(s)=\min\{t\geq 0: e^{Z_t}=s\}$.
Let 
\[
{\tilde \tau} 
 = \min\big\{t> \tau^Z(\textstyle\frac{5}{3}r) 
  : e^{Z_t}
  = \textstyle\frac{3}{2}r\big\}.
\]
Define the event (cf.\ Figure~\ref{f_go_around}; note that what is
written in the first line guarantees that the first crossing of
$\{z\in\C: \textstyle\frac{4}{3}r\leq|z|\leq\textstyle\frac{5}{3}r\}$
is not ``too wide'')
\begin{align*}
 U'_r 
 &= \Big\{{\tilde \tau}
 <\tau^Z(2r), 
 |B_{s}-B_{\tau^Z(\frac{4}{3}r)}|
\leq \textstyle\frac{1}{3}\pi \text{ for all }
s\in [\tau^Z(\frac{4}{3}r), {\tilde \tau}
],
\\
& ~~~~~~
|B_{{\tilde \tau}+1}-B_{{\tilde \tau}}|
\geq \textstyle\frac{8}{3}\pi,
\textstyle\frac{4}{3}r< e^{Z_s}
< \textstyle\frac{5}{3}r
 \text{ for all }
s\in [{\tilde \tau},
{\tilde \tau}+1]\Big\},  
\end{align*}
and observe that $\IP_x[U_{|x|}]\geq\IP_x[U'_{|x|}]$
(given that $|\hW_0|=r$, on~$U'_r$ the origin is indeed 
disconnected from the infinity). 
\begin{figure}
\begin{center}
\includegraphics[width=0.59\textwidth]{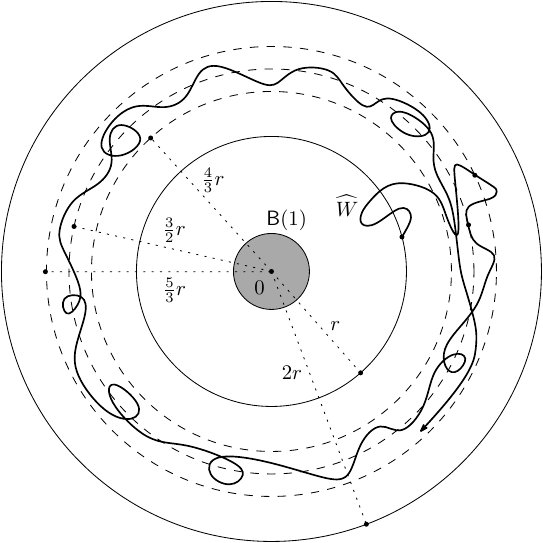}
\caption{On the definition of the event $U'_r$}
\label{f_go_around}
\end{center}
\end{figure}
By independence of~$Z$ and~$B$    
it can be easily seen that, 
for any~$y$ such that $|y|=\frac{4}{3}r$
\[
\IP_y\big[{\tilde \tau}
<\tau^Z(2r), 
 |B_{s}-B_0|
\leq \textstyle\frac{1}{3}\pi \text{ for all }
s\leq {\tilde \tau}
\big]
\geq \gamma_1,
\]
and for any~$y$ such that $|y|=\frac{3}{2}r$
\[
\IP_y\big[|B_{1}-B_{0}| \geq \textstyle\frac{8}{3}\pi,
\textstyle\frac{4}{3}r< e^{Z_s} 
< \textstyle\frac{5}{3}r
 \text{ for all } s\in [0,1]\big]\geq \gamma'_1,
\]
for some $\gamma_1,\gamma'_1>0$. We conclude the proof
using the (strong) Markov property.
\end{proof}
For $A\subset\C$, define $\rad(A)=\sup_{x\in A}|x|$.
The above result immediately implies the following fact 
for the Brownian amoeba~$\am$: 
$\IP[\rad(\am)>2^k]\leq (1-\gamma_0)^k$
for any positive integer~$k$. From this, it is 
straightforward to obtain the following result.
(it is similar to Lemma~3 of~\cite{Wer94}):
\begin{cor}
\label{c_rad_amoeba}
 There is a positive constant $\gamma_2$
such that, for all $u > 1$
\begin{equation}
\label{eq_rad_amoeba}
 \IP[\rad(\am)>u]\leq 2 u^{-\gamma_2}.
\end{equation}
\end{cor}

\section{Proofs of the main results}
\label{s_proofs}
\begin{proof}[Proof of Theorem~\ref{t_conv_Poisson}]
For $0<a<b<\infty$,
let $\Lambda_{x,\alpha}(a,b)$ be the set of BRI's trajectories
that intersect~$\B(x,b)$ but do not 
intersect~$\B(x,a)$. 
By construction, it holds that
$\Lambda_{x,\alpha}(a,b)$
and $\Lambda_{x,\alpha}(c,d)$ are independent whenever
$(a,b]\cap (c,d] = \emptyset$, and
the number of trajectories in $\Lambda_{x,\alpha}(a,b)$
has Poisson distribution with parameter
$\pi\alpha\big(\hcapa(\B(x,b))-\hcapa(\B(x,a))\big)$.

 For $b>0$, abbreviate
$r_b:= \exp\big(-\frac{2\alpha\ln^2 |x|}{b}\big)$. 
We have \[\IP[\tY_x^{(1)}(\alpha)>s]
=\exp\big(-\pi\alpha\hcapa(\B(x,r_s))\big)\] and,
 for $j\geq 1$ (and with $b=b_1+\cdots+b_j$)
we can write
\begin{align}
\lefteqn{
\IP\big[\tY_x^{(j+1)}(\alpha)>s
  \mid \tY_x^{(1)}(\alpha)=b_1,\ldots,\tY_x^{(j)}(\alpha)=b_j\big]
}
\nonumber\\
 &= \IP\big[\Lambda_{x,\alpha}(r_b,r_{b+s})
  =\emptyset\big]\nonumber\\
 &= \exp\big(-\pi\alpha\big(\hcapa(\B(x,r_{b+s}))-\hcapa(\B(x,r_b))\big)\big).\nonumber
 \label{nobody_in_Lambda}
\end{align}
Next, we use Lemma~\ref{l_cap_smalldisk} to obtain that
\begin{equation}
\label{cap_B_x_r_h}
\pi\alpha \hcapa(\B(x,r_h))
  = h\big(1 + O\big(
  \Psi_1^{(h)} + \Psi_2^{(h)} + \Psi_3^{(h)}
    \big)\big).
\end{equation}
Using this in the above calculations 
together with the fact that 
(for bounded~$\Omega$) 
\[e^{-v(1+O(\Omega))}= e^{-v}(1+vO(\Omega)),\]
it is straightforward to
conclude the proof of Theorem~\ref{t_conv_Poisson}.
\end{proof}

\begin{proof}[Proof of Theorem~\ref{t_amoeba}]
For the proof of part~(i), 
denote $\am_1$ the connected component of the origin
formed by the closest trajectory (the one at distance~$\rho_1^\alpha$);
it is clear that $\am_1 \eqlaw \rho_1^\alpha \am$
 (with~$\am$ being a Brownian amoeba
independent of $(\rho^\alpha_k, k\geq 1)$).
In the following, we obtain an upper bound for the 
probability that another trajectory would ``touch'' the 
cell formed by the first trajectory.
We have (recall~\eqref{rho_exponential}) 
$\rho_2^\alpha/\rho_1^\alpha = \exp(Y_2/(2\alpha))$, 
where~$Y_2$ is an Exponential(1) random variable,
independent of~$\rho_1^\alpha$.
Therefore, 
we can write
\begin{align}
 \IP[\cl_0^\alpha = \am_1] &\geq 
 \IP\Big[\rad(\am_1) < \rho_2^\alpha\Big]\nonumber\\
 &= \IP\Big[\rho_1^\alpha\rad(\am) 
 < \rho_2^\alpha\Big]\nonumber\\
 &=\IP\Big[\rad(\am)
  <  \exp\Big(\frac{Y_2}{2\alpha}\Big)\Big]\nonumber\\
&= \IE \Big(\IP\Big[\rad(\am)
  <  \exp\Big(\frac{Y_2}{2\alpha}\Big)\; \Big| 
  \; Y_2\Big]\Big)\nonumber\\
  \intertext{\qquad \qquad \qquad \qquad \qquad 
\footnotesize (by Corollary~\ref{c_rad_amoeba})}
 & \geq 1- 2\IE \exp\Big(-\frac{\gamma_2 Y_2}{2\alpha}\Big)
 \nonumber\\
 &= 1- 2\frac{1}{1+\frac{\gamma_2}{2\alpha}}\nonumber\\
 &= 1- O(\alpha) \qquad \text{ as }\alpha\to 0,
\label{calc_amoeba_included}
\end{align}
and this completes
the proof of the part~(i).

Let us now prove the part~(ii).
Recall the notation $\Lambda_{x,\alpha}(a,b)$ from 
the beginning of the proof
of Theorem~\ref{t_conv_Poisson}.
Considering events in the sequence
$\big(\Lambda_{x,\alpha}(e^{-n},e^{-(n-1)}),n\in\Z\big)$ 
are independent,
let us define the events
\[
 \Upsilon_n 
  = \big\{\Phi_x(\alpha)\in (e^{-n},e^{-(n-1)}]\big\},
\]
and let 
$\sigma_{x,\alpha}=\lfloor -\ln\Phi_x(\alpha)\rfloor +1$,
i.e., $\sigma_{x,\alpha}$ is the only integer~$k$
for which~$\Upsilon_k$ occurs. 
Let $\hW^{x,(n)}$ be the bi-infinite
trajectory defined in the following way:
let $\hW^{x,(n)}_0$ be chosen uniformly
at random on~$\partial\B(x,e^{-(n-1)})$;
then $(\hW^{x,(n)}_t, t<0)$ is distributed
as $\hW^{x,e^{-(n-1)}}$ and 
$(\hW^{x,(n)}_t, t>0)$ is distributed
as $\hW^{x,e^{-n}}$. That is, from a
 random point on~$\partial\B(x,e^{-(n-1)})$
we draw one trajectory conditioned on (immediately) escaping
 from~$\B(x,e^{-(n-1)})$ and another one 
conditioned on never hitting~$\B(x,e^{-n})$.
Then, Lemma~3.9 of~\cite{CP20} implies that
the cell~$\am^{x,(n)}$ 
formed by $\hW^{x,(n)}$ around~$x$
is the standard Brownian amoeba 
rescaled by $\inf_{t>0}\big|\hW^{x,(n)}_t-x\big|$,
and then shifted to~$x$. 

An important observation is that
all $\BRI$'s trajectories are actually
ordered by their $\alpha$-coordinates
(see e.g.\ Remark~2.7 of~\cite{CP20}),
so it may make sense to speak about
the \emph{first} trajectory belonging
to some set of trajectories
(in the case when we are able to find one 
with minimal $\alpha$-coordinate among them).
By Lemma~\ref{l_cond_full_B(y,r)}
(and also assuring that the entrance 
measure of that $\BRI$'s trajectory 
to~$\partial\B(x,e^{-(n-1)})$ is not far from uniform),
we can have a coupling $\hW^{x,(n)}$ with 
the first $\BRI$'s
trajectory (for all $\alpha>0$) that 
intersects~$\B(x,e^{-(n-1)})$ but does not 
intersect~$\B(x,e^{-n})$,
such that the traces of these coincide with high 
probability in the vicinity of~$\B(x,e^{-n})$
(more precisely, in $\B(x,e^{-n}n^\theta)$,
where $\theta:=\max(2,\gamma_2^{-1})$ with~$\gamma_2$
of Corollary~\ref{c_rad_amoeba}).
Now, we are interested in proving
that, with high probability,
$\cl_x(\alpha)=\am^{x,(\sigma_{x,\alpha})}$.
Let us define three families of events
\begin{align*}
 G_1^{(n)} &= \big\{\am^{x,(n)} 
     \subset \B(x,e^{-n}n^\theta)\big\},
 \\
  G_2^{(n)} &= \big\{\Phi_x^{(2)}(\alpha) 
    > e^{-n}n^\theta \big\},
 \\
  G_3^{(n)} &= 
\big\{\hW^{x,(n)}_{[-\infty,\infty]}
\cap\B(x,e^{-n}n^\theta)
= \hW^{*,(n)}_{[-\infty,\infty]}
\cap\B(x,e^{-n}n^\theta)\big\},
\end{align*}
where~$\hW^{*,(n)}$ stands 
for the first $\BRI$'s
trajectory 
that 
intersects~$\B(x,e^{-(n-1)})$ but does not 
intersect~$\B(x,e^{-n})$;
as suggested above,
for each~$n$, we assume that it is coupled with the
$\hW^{x,(n)}$ in such a way 
that the probability
that the corresponding traces coincide is maximized.

Recall the notation $r_b=\exp\big(-2\alpha b^{-1}\ln^2 |x|\big)$,
so that $\ln r_b^{-1}=2\alpha b^{-1}\ln^2 |x|$.
Also, abbreviate $b_0:=3\ln (\alpha\ln^2|x|)>1$ 
and $n_0=\ln r_{b_0}^{-1}=\frac{2\alpha\ln^2 |x|}
{3\ln (\alpha\ln^2|x|)}$.
Note that Theorem~\ref{t_conv_Poisson} implies that 
\begin{equation}
\label{tY>b0}
 \IP\big[\sigma_{x,\alpha}\leq n_0\big] 
 \leq O\big(\tfrac{1}{(\alpha \ln^2|x|)^3}\big).
\end{equation}
Then, 
Corollary~\ref{c_rad_amoeba} implies that, 
for $n\geq n_0$
\begin{equation}
\label{prob_G1_rb}
 \IP\big[G_1^{(n)}\big] 
 \geq 1-2n^{-\theta \gamma_2} 
 \geq 1-2n^{-1}  
  \geq 1-\frac{3\ln (\alpha\ln^2|x|)}{\alpha \ln^2 |x|}.
\end{equation}
 
Observe that if~$N$ is a Poisson-distributed
random variable, then an elementary computation implies that
$\IP[N=1 \mid N \geq 1] \geq \IP[N=0]$.
Therefore, quite analogously to the 
proof of Theorem~\ref{t_conv_Poisson},
using~\eqref{cap_B_x_r_h} after some elementary
calculations we can obtain for $n\geq n_0$
(note that $\Psi_1^{(h)}$ and $\Psi_3^{(h)}$
grow in~$h$ and $\Psi_2^{(h)}$ grows in~$h$
at least when $\alpha\ln^2|x|$ is large enough)
\begin{align}
\lefteqn{
 \IP\big[G_2^{(n)}\mid \Upsilon_n\big] 
 }
\nonumber\\ 
 &= \IP\big[\mathop{\rm card}\big(\Lambda_{x,\alpha}
 (e^{-n},e^{-(n-1)})\big)
  = 1 \mid \Upsilon_n\big]
  \IP\big[\Lambda_{x,\alpha}(e^{-(n-1)},e^{-n}n^\theta)
  = \emptyset\big]
    \nonumber\\
&\geq \IP\big[\Lambda_{x,\alpha}(e^{-n},e^{-(n-1)})
  = \emptyset\big]
\IP\big[\Lambda_{x,\alpha}(e^{-(n-1)},e^{-n}n^\theta)
  = \emptyset\big]
  \nonumber\\
  &= \IP\big[\Lambda_{x,\alpha}(e^{-n},e^{-n}n^\theta)
  = \emptyset\big]
    \nonumber\\
  & \geq 1-O\big(\Psi^* + 
  b_0(H_1+H_2+H_3)
     \big),
\label{prob_G2_rb} 
\end{align}
where
\begin{align*}
 \Psi^* &= \frac{\ln^3 (\alpha\ln^2|x|)}{\alpha\ln^2|x|},\\
 H_1 &= \frac{\tell_x \ln (\alpha\ln^2|x|)}{\alpha\ln^2|x|},\\
 H_2 &= \frac{\alpha\ln^2|x|+\ln|x|\times \ln(\alpha\ln^2|x|)}
  {|x|\exp\big(\frac{2\alpha\ln^2|x|}
  {3\ln(\alpha\ln^2|x|)}\big)\ln|x|},\\
 H_3 &= \frac{1+H_1}{(|x|-1)\exp\big(\frac{2\alpha\ln^2|x|}
  {3\ln(\alpha\ln^2|x|)}\big)}.
\end{align*}
We now intend to simplify the term 
$O(\Psi^* + b_0(H_1+H_2+H_3))$.
First, it can be easily seen that, as $\alpha\ln^2|x|\to\infty$,
the term~$H_2$ is negligible in comparison to~$H_1$.
Then, assuming that~$H_1\leq 1$,
we have
\[
 H_3\leq \frac{2}{(|x|-1)\exp\big(\frac{\alpha\ln^2|x|}  {3\ln(\alpha\ln^2|x|)}\big)}
   \times \exp\Big(-\frac{\alpha\ln^2|x|}
  {3\ln(\alpha\ln^2|x|)}\Big);
\]
since
\begin{equation}
\label{denominator>1} 
 \ln\Big((|x|-1)\exp\big(\frac{\alpha\ln^2|x|}
  {3\ln(\alpha\ln^2|x|)\big)}\Big)
   \geq \frac{\alpha\ln^2|x|}{3\ln(\alpha\ln^2|x|)\big)}
  - |\ln(|x|-1)| \geq 0,
\end{equation}
if 
$
H_1$
is small enough, we see that~$H_3$ is at most of order of~$H_1$.
Therefore, since both $\Psi^*$ and~$b_0H_1$ are 
bounded from above by 
$\tfrac{\tell_x\ln^3 (\alpha\ln^2|x|)}{\alpha\ln^2|x|}$,
for $n\geq n_0$ we obtain that
\begin{equation}
\label{prob_G2_rb_refined}  
  \IP\big[G_2^{(n)}\mid \Upsilon_n\big] 
   \geq 1 - O\big(\tfrac{\tell_x \ln^3(\alpha\ln^2|x|)}
   {\alpha\ln^2|x|}\big).
\end{equation}

Next, we deal with the event~$G_3^{(n)}$.
To achieve a successful coupling of~$\hW^{x,(n)}$
with the actual $\BRI(\alpha)$ trajectory,
we first have to estimate the probability
that their entrance points 
to~$\partial\B(x,e^{-(n-1)})$ can be 
successfully coupled. Here substituting for the moment $r:=e^{-(n-1)}$
and recall Lemma~3.10 of~\cite{CP20}: the entrance measure
of a $\BRI$'s trajectory to $\B(x,r)$ is 
the harmonic (i.e., uniform) 
measure~$\hm_{\B(x,r)}$ biased by the logarithm 
of the norm of the entrance point
(and obviously that, it holds  
$\frac{\ln|y|}{\ln|z|}=(1+O(\frac{r}{|x|\ln|x|}))$ for $y,z\in\partial\B(x,r)$).
Then, the entrance measure of a trajectory
of $\Lambda_{x,\alpha}(r/e,r)$ is the preceding
one additionally biased by the probability
of not hitting $\B(x,r/e)$.
We note that, for $y\in\partial\B(x,r)$,
Lemma~\ref{l_cond_escape_or_hit} gives us that
\[
 \IP_y[\htau(x,r/e)=\infty]
  = \frac{1}{\ln\frac{e}{r}+\ell_x+\ln|x|}
  \Big(1+O\big(\tfrac{r\tell_x}
{(|x|-1)(\ln r^{-1}+\ln|x|)^2}\big)\Big);
\]
this implies that for any 
$y,z\in\partial\B(x,r)$,
\begin{equation}
\label{yz_reach_inside}
 \bigg|\frac{\IP_y[\htau(x,r/e)=\infty]}
 {\IP_z[\htau(x,r/e)=\infty]}-1\bigg|
 \leq O\big(\tfrac{r\tell_x}
{(|x|-1)(\ln r^{-1}+\ln|x|)^2}\big).
\end{equation}
Therefore, the probability of successfully
coupling the entrance points is at least
\begin{equation}
\label{some_Os}
1- O\big(\tfrac{r}{|x|\ln|x|}+
\tfrac{r\tell_x}
{(|x|-1)(\ln r^{-1}+\ln|x|)^2}\big)
 = 1- O\big(\tfrac{r\tell_x}
{(|x|-1)(\ln r^{-1}+\ln|x|)^2}\big).
\end{equation}
Next, to couple the actual traces
with Lemma~\ref{l_cond_full_B(y,r)},
we have to assure that~\eqref{r_much_smaller_y-1} 
is verified for $n\geq n_0$.
It is clear that it holds for large enough $\alpha\ln^2|x|$
in the case $|x|\geq 2$, since it would be enough
to check that\footnote{Observe that the function
$f(x)=x\ln^{2\theta}x^{-1}$ is increasing for 
$x\in(0,e^{-2\theta})$.} 
$r_{b_0}\ln^{2\theta}r_{b_0}^{-1}\leq \frac{1}{2}$,
which is clearly
the case. Assume now that $|x|<2$;
essentially, we need to verify that 
$r_{b_0}\ln^{2\theta}r_{b_0}^{-1}<(|x|-1)^{1+\delta_0}$.
Let $\alpha\ln^2|x|$ be large enough so that
$r_{b_0}\ln^{2\theta}r_{b_0}^{-1}\leq r_{b_0}^{1/2}$.
Then, as in~\eqref{denominator>1}, we have 
\begin{equation}
\label{once_again_the_same}
\ln\frac{(|x|-1)^{1+\delta_0}}{r_{b_0}^{1/2}}
 \geq \frac{\alpha\ln^2|x|}{3\ln(\alpha\ln^2|x|)\big)}
  -(1+\delta_0) \big|\ln(|x|-1)\big| > 0
\end{equation}
if $\frac{\tell_x \ln (\alpha\ln^2|x|)}{\alpha\ln^2|x|}$
is small enough, so~\eqref{r_much_smaller_y-1} 
indeed holds.

Then, using Lemma~\ref{l_cond_full_B(y,r)} (for 
both positive and negative sides of the trajectories,
which are independent except for their starting point),
we obtain for $n\geq n_0$ (recall~\eqref{some_Os})
\begin{equation}
\label{prob_G3_rb}
 \IP\big[G_3^{(n)}\mid \Upsilon_n\big] 
  \geq 1 - O\big(\tfrac{\ln(\alpha b_0^{-1}\ln^2 |x|)}
  {\ln|x|+ \alpha b_0^{-1}\ln^2 |x|}
  + \tfrac{\tell_x r_{b_0}\ln^{2\theta}r_{b_0}^{-1}}{|x|-1} 
  + \tfrac{r\tell_x}
{(|x|-1)(\ln r^{-1}+\ln|x|)^2}\big)
\end{equation}
(again, it is not difficult to see that the second 
term with generic $b\in(0,b_0)$ is majorized 
by the term with~$b_0$, in the case when 
$\alpha\ln^2|x|$ is large enough).
We have that
\[\frac{\ln(\alpha b_0^{-1}\ln^2 |x|)}
  {\ln|x|+ \alpha b_0^{-1}\ln^2 |x|}
\leq 
\frac{3\ln^2(\alpha \ln^2 |x|)}{\alpha \ln^2 |x|}.\]
As for the second term, it is clearly of smaller 
order when~$|x|\geq 2$. If~$1<|x|< 2$, 
we have $\tell_x\leq O\big((|x|-1)^{-1}\big)$
so we can write
\[
 \frac{\tell_x r_{b_0}\ln^{2\theta}r_{b_0}^{-1}}{|x|-1}
  \leq O\Big(\frac{r_{b_0}^{1/2}}{(|x|-1)^2}\Big)
  \leq r_{b_0}^{1/4}
    \times O\Big(\frac{r_{b_0}^{1/4}}{(|x|-1)^2}\Big);
\]
analogously to~\eqref{denominator>1} and~\eqref{once_again_the_same},
we can show that $\frac{r_{b_0}^{1/4}}{(|x|-1)^2}\leq 1$
(assuming that
$\frac{\tell_x \ln (\alpha\ln^2|x|)}{\alpha\ln^2|x|}$
is small enough),
which allows us to conclude that
the second term 
is of a smaller order than the first one in that case as well. The third term can be treated quite analogously 
to the second one, so we obtain
\begin{equation}
\label{prob_G3_rb_refined}
 \IP\big[G_3^{(n)}\mid \Upsilon_n\big] 
\geq 1 - O\big(\tfrac{\tell_x \ln^2(\alpha\ln^2|x|)}{\alpha\ln^2|x|}
 \big)
\end{equation}
for $n\geq n_0$.

We can now estimate the probability of the event of interest.
It remains to write
using \eqref{tY>b0}, \eqref{prob_G1_rb}, \eqref{prob_G2_rb_refined},
\eqref{prob_G3_rb_refined}
\begin{align*}
 \IP\big[\cl_x(\alpha)=\am^{x,(\sigma_{x,\alpha})}\big]
 &=
 \IE\big(\IP\big[\cl_x(\alpha)
 =\am^{x,(\sigma_{x,\alpha})}
    \mid \sigma_{x,\alpha} \big]\big)\\
  &\geq  
 \IE\big(\IP\big[\cl_x(\alpha)
  =\am^{x,(\sigma_{x,\alpha})}\mid \sigma_{x,\alpha}\big]
  \1{\sigma_{x,\alpha}\geq n_0}\big)\\
  &\geq \IP\big[\sigma_{x,\alpha}\geq n_0\big] \times
    \min_{n\geq n_0}
 \IP\big[G_1^{(n)}\cap G_2^{(n)} \cap G_3^{(n)}\mid
 \Upsilon_n\big]\\
 &\geq 1 - O\big(\tfrac{\tell_x \ln^3(\alpha\ln^2|x|)}
   {\alpha\ln^2|x|}\big).
\end{align*}
This concludes the proof of Theorem~\ref{t_amoeba}.
\end{proof}

\begin{proof}[Proof of Theorem~\ref{t_central_cell}]
We need to prove that for any $\delta>0$
it holds that
\begin{equation}
\label{central_cell_a.s.} 
(1-2\delta)\sqrt{\frac{\ln \alpha}{2\alpha}}
  \leq M_\alpha \leq
 (1+2\delta)\sqrt{\frac{\ln \alpha}{2\alpha}}
\end{equation}
eventually for all large enough~$\alpha$.
 
 We start by showing that the second inequality 
in~\eqref{central_cell_a.s.}
(with $\delta$ on the place of $2\delta$) 
is verified with high probability.
Let us place $n_{\delta,\alpha}:= 
\big\lceil \frac{6\pi}{\delta}\sqrt{\frac{2\alpha}{\ln\alpha}} \big\rceil$
points $x_1,\ldots, x_{n_{\delta,\alpha}}$
on~$\partial\B(1)$ in such a way that every closed arc 
of length 
$\frac{\delta}{3}\sqrt{\frac{\ln\alpha}{2\alpha}}$
contains at least one of these points.
Then, we have the following inclusion:
\begin{equation}
\label{incl_M_alpha_Phi_x}
 \Big\{M_\alpha>(1+\delta)\sqrt{\frac{\ln\alpha}{2\alpha}}
\Big\}
  \,\subset \,\,
 \bigcup_{k=1}^{n_{\delta,\alpha}}
   \Big\{\Phi_{x_k}(\alpha)>\Big(1+\frac{\delta}{3}\Big)
   \sqrt{\frac{\ln\alpha}{2\alpha}}\Big\}.
\end{equation}
Next, we recall that Lemma~3.12 of~\cite{CP20}
states that, for $x\in\partial\B(1)$
and small enough~$r>0$,
\begin{equation}
\label{cap_blister}
 \hcapa\big(\B(x,r)\big) = \capa\big(\B(1)\cup\B(x,r)\big)
  = \frac{r^2}{\pi}\big(1+O(r)\big),
\end{equation}
so, by~\eqref{eq_vacant_Bro}
\[
 \IP\Big[\Phi_{x_k}(\alpha)
 >\Big(1+\frac{\delta}{3}\Big)
   \sqrt{\frac{\ln\alpha}{2\alpha}}\Big]
  = \exp\Big(-\frac{(1+\frac{\delta}{3})^2
  \ln \alpha}{2}
  \big(1+o(1)\big)\Big)
  \leq O\big(\alpha^{-\frac{1}{2}
  (1+\frac{2\delta}{3})}
  \big);
\]
by the union bound, this shows that
\begin{equation}
\label{M_alpha>}
 \IP\Big[M_\alpha>(1+\delta)
 \sqrt{\frac{\ln\alpha}{2\alpha}}
\Big] \leq O\big(\alpha^{-\delta/3}\big).
\end{equation}

 Now, we need to obtain an upper-bound
on the probability that~$M_\alpha$ does not 
exceed $(1-\delta)\sqrt{\frac{\ln\alpha}{2\alpha}}$.
Denote $m_{\delta,\alpha}= 
\big\lceil \alpha^{\frac{1}{2}(1-\frac{\delta}{3})}
\big\rceil$, and 
let $y_1,\ldots,y_{m_{\delta,\alpha}}$ be points
placed on~$\partial\B(1)$ with equal spacing 
between neighbouring ones
(so that the distance between the neighbouring ones
is of order~$\alpha^{-\frac{1}{2}(1-\frac{\delta}{3})}
$). Again, by~\eqref{eq_vacant_Bro}
and~\eqref{cap_blister} for large enough~$\alpha$
we have that for any~$k$
\begin{align}
 \IP\Big[\Phi_{y_k}(\alpha)\leq (1-\delta)
   \sqrt{\frac{\ln\alpha}{2\alpha}}\Big]
  &= 1-\exp\Big(-\frac{(1-\delta)^2
  \ln \alpha}{2}
  \big(1+o(1)\big)\Big)
\nonumber\\  
 & \leq 1-\alpha^{-\frac{1}{2}(1-\delta)}.
\label{prob_come_to_yk}
\end{align}

Next, abbreviate 
$A_k=\B\big(y_k,(1-\delta)
\sqrt{\frac{\ln\alpha}{2\alpha}}\big)\setminus\B(1)$,
and consider all the interlacement trajectories
that has an intersection with 
$\bigcup_{k=1}^{m_{\delta,\alpha}}A_k$.
We say that such a trajectory is of type~$1$
if it intersects only one of the $A_k$'s, and 
 of type~$2$ if it intersects several of them.
By construction, the numbers of type-$1$ 
trajectories that intersect~$A_k$ are 
\emph{independent} Poisson random variables
for $k=1,\ldots,m_{\delta,\alpha}$,
and they are also independent of the process 
of type-$2$ trajectories.
Note also that, by~\eqref{cap_blister},
the total number of trajectories (i.e., of both
types) that intersect~$A_k$ is Poisson
with parameter $\frac{1}{2}(1-\delta)^2
  \ln \alpha\times\big(1+o(1)\big)$;
this means that the number of type-$1$ 
trajectories that intersect~$A_k$ is 
dominated by a Poisson random variable
with that parameter.
Therefore, from~\eqref{prob_come_to_yk}
we obtain
\begin{align}
\lefteqn{
 \IP\big[\text{there exists }j\text{ such that 
 no type-$1$ trajectory hits }A_j
 \big] 
  } ~~~~~~~~~~~~~~~~~~~~~~~~
\nonumber\\ 
 &\geq
 1-\big(1-\alpha^{-\frac{1}{2}
 (1-\delta)}\big)^{
 m_{\delta,\alpha}}
\nonumber\\ 
 &\geq
1- \exp(-\alpha^{\delta/3}).
\label{there_is_untouched_A}
\end{align}

Now, we work with the process 
of type-$2$ trajectories. First, we claim that,
for any~$k$ and any $x\in A_k$
\begin{equation}
\label{hit_another_A}
 \IP_x\Big[\hW \text{ hits }
 \bigcup_{\ell\neq k}A_\ell\Big]
 \leq O\big(\alpha^{-\delta/3}\ln\alpha\big).
\end{equation}
To show the above, we use
Proposition~\ref{p_represent_hW} together
with the fact that the Bes(3) process can be 
represented as the norm of the $3$-dimensional
Brownian motion. Assume without restricting
generality that~$k=1$ in~\eqref{hit_another_A},
and that $y_1=1$. Denote 
$r_{\delta,\alpha}=(1-\delta)
 \sqrt{\frac{\ln\alpha}{2\alpha}}$
and let~$A'_j$ be the pre-image of~$A_j$ with 
respect to the exponential map. It is clear that,
for fixed~$\delta$ and large enough~$\alpha$ 
we have 
$A'_j\subset [0,2r_{\delta,\alpha}]\times I_j$,
where~$I_j$ is the interval of length 
$4r_{\delta,\alpha}$ centred at the pre-image 
of~$y_j$. For $j\in\Z$, denote also 
\[
 B'_j = [0,2r_{\delta,\alpha}]
  \times \Big[\frac{2\pi}{m_{\delta,\alpha}}j
  -2r_{\delta,\alpha}, 
  \frac{2\pi}{m_{\delta,\alpha}}j
  +2r_{\delta,\alpha}\Big]
  \subset \R_+\times \R.
\]
Let~$\varphi:\R^4\to \R_+\times \R$ be defined
by $\varphi(x,y,z,t)=(\sqrt{x^2+y^2+z^2},t)$,
and let~$W^{(4)}$ be the standard Brownian motion
in~$\R^4$ started at the origin.
We then have, for $x'\in A'_1$
\begin{align*}
 \IP_{x'}\Big[(Z,B) \text{ hits }
 \bigcup_{2\leq \ell\leq m_{\delta,\alpha}}A'_\ell\Big] 
  &\leq 
\IP_{\varphi^{-1}(x')}\Big[
W^{(4)} \text{ hits }
 \bigcup_{\ell\in \Z\setminus \{0\}}\varphi^{-1}(B'_\ell)\Big].
\end{align*}
Recall that, in $\R^4$, we have 
(momentarily ``lifting'' the notation 
$\B(\cdot,\cdot)$ to that dimension)
$\IP_x[W^{(4)} \text{ hits }\B(y,r)]\asymp
\big(\frac{r}{|x-y|}\big)^2$.
Then, $\varphi^{-1}(B'_\ell)$ is a cylinder
of the linear size of the order
$\sqrt{\frac{\ln\alpha}{\alpha}}$
(so it fits into a ball of the size of the same order),
and the distance between~$\varphi^{-1}(x')$
and $\varphi^{-1}(B'_\ell)$ is of order
$|\ell| \alpha^{-\frac{1}{2}(1-\frac{\delta}{3})}$.
This indeed shows that~\eqref{hit_another_A} holds
(note that $\sqrt{\frac{\ln\alpha}{\alpha}}/
 \alpha^{-\frac{1}{2}(1-\frac{\delta}{3})}
 = \alpha^{-\delta/6}\sqrt{\ln\alpha}$, and 
we also have to use the fact that
$\sum_{\ell\in \Z\setminus \{0\}}|\ell|^{-2}<\infty$).

Now, \eqref{hit_another_A} shows that 
the number of type-$2$ trajectories
that hits a set~$A_k$ has expectation
at most $O\big(\alpha^{-\delta/3}\ln^2\alpha\big)$
(recall that the total number of trajectories
hitting~$A_k$ is Poisson 
with parameter~$O(\ln\alpha)$). Using the fact 
that the process of type-$2$ trajectories is 
independent from all the type-$1$ ones, 
we find that 
if~$\zeta\in\{1,\ldots,m_{\delta,\alpha}\}$ is a (random) index such that no type-$1$ trajectory
touches~$A_{\zeta}$ (by~\eqref{there_is_untouched_A},
such~$\zeta$ exists with probability
at least 
$1-\exp(-\alpha^{\delta/3})$), then 
also no type-$2$ 
trajectory touches~$A_{\zeta}$ 
with probability
at least $1-O\big(\alpha^{-\delta/3}\ln^2\alpha\big)$.
This means that
\begin{equation}
\label{M_alpha<}
 \IP\Big[M_\alpha<(1-\delta)
 \sqrt{\frac{\ln\alpha}{2\alpha}}
\Big] 
\leq O\big(\alpha^{-\delta/3}\ln^2\alpha\big).
\end{equation}
Together with~\eqref{M_alpha>},
this implies the convergence in probability
in~\eqref{eq_M_central_cell}; now, we will
deduce the a.s.\ convergence using the 
monotonicity of~$M_\alpha$.
Indeed, let $\alpha_k=e^{\sqrt{k}}$. By the  
Borel-Cantelli lemma,
for any fixed $\delta>0$ the estimates~\eqref{M_alpha>}
and~\eqref{M_alpha<} imply that 
\begin{equation}
\label{conv_alpha_k}
(1-\delta)
 \sqrt{\frac{\ln\alpha_k}{2\alpha_k}} \leq 
 M_{\alpha_k} \leq (1+\delta)
 \sqrt{\frac{\ln\alpha_k}{2\alpha_k}}
\end{equation}
for all but finitely many~$k$.
We have that 
$M_{\alpha_{k+1}}\leq M_\alpha\leq M_{\alpha_k}$
for all $\alpha\in [\alpha_k,\alpha_{k+1}]$,
and it is also elementary to verify that
\[
 \frac{\sqrt{\frac{\ln\alpha_{k+1}}{2\alpha_{k+1}}}}{\sqrt{\frac{\ln\alpha_k}{2\alpha_k}}}
 \to 1 \text{ as }k\to\infty.
\]
This allows us to deduce~\eqref{central_cell_a.s.}
 from~\eqref{conv_alpha_k}
and therefore concludes the proof of
Theorem~\ref{t_central_cell}.
\end{proof}

\section*{Acknowledgements}
The authors were partially supported by
CMUP, member of LASI, which is financed by national funds
through FCT --- Funda\c{c}\~ao
para a Ci\^encia e a Tecnologia, I.P., 
under the project with reference UIDB/00144/2020.
The authors thank the anonymous referees for 
their helpful comments and suggestions
(in particular, for observing that one can obtain
the a.s.\ convergence in Theorem~\ref{t_central_cell}
using the monotonicity of~$M_\alpha$).


{\small\bibliography{cimart}}

\EditInfo{October 15, 2024}{January 30, 2025}{Ivan Kaygorodov}

\end{document}